\documentclass[a4paper,oneside]{article}

\usepackage{amsmath,amssymb,amsthm}
\usepackage{fullpage}
\usepackage{enumitem}
\usepackage{mathtools}
\usepackage[dvipdfmx]{graphicx,xcolor}

\theoremstyle{definition}
\newtheorem{thm}{Theorem}[subsection]
\newtheorem{prp}[thm]{Proposition}
\newtheorem{cor}[thm]{Corollary}
\newtheorem{lem}[thm]{Lemma}

\newtheorem{fct}[thm]{Fact}
\newtheorem{rmk}[thm]{Remark}
\newtheorem{eg}[thm]{Example}
\newtheorem{mth}{Theorem}

\numberwithin{equation}{subsection}
\numberwithin{table}{subsection}
\numberwithin{figure}{subsection}
\allowdisplaybreaks
\newcommand{\oa}{\overrightarrow}
\newcommand{\ol}{\overline}
\newcommand{\wh}{\widehat}
\newcommand{\wt}{\widetilde}
\newcommand{\ep}{\epsilon}

\newcommand{\hf}{\frac{1}{2}}
\newcommand{\bgeq}{\succcurlyeq_B}
\newcommand{\bleq}{\preccurlyeq_B}

\newcommand{\tprod}{\textstyle \prod}
\newcommand{\tbcup}{\textstyle \bigcup}

\newcommand{\bN}{\mathbb{N}}
\newcommand{\bZ}{\mathbb{Z}}
\newcommand{\bQ}{\mathbb{Q}}
\newcommand{\bR}{\mathbb{R}}

\newcommand{\bK}{\mathbb{K}}

\newcommand{\cL}{\mathcal{L}}

\newcommand{\frh}{\mathfrak{h}}
\newcommand{\frS}{\mathfrak{S}}

\newcommand{\wtt}{\mathrm{wt}}

\DeclareMathOperator{\sh}{sh}

\DeclareMathOperator{\GL}{GL}

\DeclareMathOperator{\des}{des}
\DeclareMathOperator{\hgt}{ht}
\DeclareMathOperator{\wgt}{wt}
\DeclareMathOperator{\dir}{d}

\DeclareMathOperator{\End}{End}

\newcommand{\abs}[1]{\left|{#1}\right|}
\newcommand{\br}[1]{\langle{#1}\rangle}
\newcommand{\pr}[1]{\left\{#1\right\}}
\newcommand{\rst}[2]{\left.#1\right|_{#2}}

\title{\textbf{A Littlewood-Richardson rule for \\ Koornwinder polynomials}}
\author{Kohei YAMAGUCHI\thanks{Graduate School of Mathematics, Nagoya University.
Furocho, Chikusaku, Nagoya, Japan, 464-8602. \newline
\qquad e-mail address: \texttt{d20003j@math.nagoya-u.ac.jp}}}
\date{2020/12/02} 

\begin{document}

\maketitle

\begin{abstract}

Koornwinder polynomials are $q$-orthogonal polynomials equipped with extra five parameters
and the $B C_n$-type Weyl group symmetry, which were introduced by Koornwinder (1992)
as multivariate analogue of Askey-Wilson polynomials.
They are now understood as the Macdonald polynomials associated to the affine root system 
of type $(C^\vee_n,C_n)$ via the Macdonald-Cherednik theory of double affine Hecke algebras. 
In this paper we give explicit formulas of Littlewood-Richardson coefficients for 
Koornwinder polynomials, i.e., the structure constants of the product as invariant polynomials.
Our formulas are natural $(C^\vee_n,C_n)$-analogue of Yip's alcove-walk formulas (2012)
which were given in the case of reduced affine root systems.
\end{abstract}

\tableofcontents

\section{Introduction}\label{s:intro}

\subsection{Koornwindder polynomials}\label{ss:intro:K}

Askey-Wilson polynomials \cite{AW} are $q$-orthogonal polynomials of one variable 
equipped with extra parameters $(a,b,c,d)$, 
which recover various $q$-analogue of Jacobi polynomials by specialization of the parameters. 
In \cite{K}, Koornwinder introduced $n$-variable analogue of Askey-Wilson polynomials,
which are today called \emph{Koornwinder polynomials}.
In the $n=1$ case they coincide with Askey-Wilson polynomials, 
and in the case of $n \ge 2$ they are equipped with extra five parameters $(a,b,c,d,t)$.
By specializing these parameters, 
one can recover Macdonald polynomials \cite{Mp,M} of $B C_n$-types.

Let us give a brief explanation on Macdonald polynomials. 
Let $S$ be an affine root system in the sense of \cite[Chap.\ 1]{M}.
If $S$ is reduced, then $S=S(R)$ or $S=S(R)^\vee$, where $S(R)$ is 
the affine root system associated to an irreducible finite root system $R$,
and $S(R)^\vee$ is the dual of $S(R)$.
In the reduced case, if $R$ is a finite root system of type $X$ 
($X=A_n$, $B_n$, $C_n$, $D_n$, $B C_n$, $E_6$, $E_7$, $E_8$, $F_4$ or $G_2$),
then we call $S(R)$ an \emph{affine root system of type $X$} and 
call $S(R)^\vee$ an \emph{affine root system of type $X^\vee$}.
We also call a reduced $S$ an \emph{untwisted affine root system}.
On the other hand, if $S$ is non-reduced, then it is of the form $S=S_1 \cup S_2$, 
where $S_1$ and $S_2$ are reduced affine root systems.
In the non-reduced case, if $S_1$ and $S_2$ are of type $X$ and $Y$ respectively, 
then we call $S$ an \emph{affine root system of type $(X,Y)$}.

The \emph{Macdonald polynomial $P_\lambda(x)$} is a $q$-orthogonal polynomial
which is a simultaneous eigenfunction of a family of $q$-difference operators 
associated to an affine root system $S$.
Today Macdonald polynomials are formulated by the \emph{Macdonald-Cherednik theory},
which is based on the representation theory of affine Hecke algebras.
This theory was first developed for untwisted affine root systems.
Below we call $P_\lambda(x)$ \emph{Macdonald polynomial of type $X$}
if the corresponding untwisted affine root system $S$ is of type $X$.

Let us go back to Koornwinder polynomials.
By the works of Noumi \cite{N}, Sahi \cite{Sa}, Stokman \cite{S} and others, 
it is clarified that one can apply Macdonald-Cherednik theory to the non-reduced
affine root system of type $(C^\vee_n,C_n)$ in the sense of \cite[Chap.\ 1]{M}, 
and that one can recover Koornwinder polynomials 
as Macdonald polynomials of type $(C^\vee_n,C_n)$.
As a result, Koornwinder polynomials are characterized as 
the ones having most parameters in the family of Macdonald polynomials.

For the convenience of the following explanation, 
let us give a brief account on the notations used in this paper.
First, we introduce the notations for the root system $R$ of type $C_n$. 
See \S \ref{sss:r:Cn} for details.  
Let $\frh_\bZ^* := \bigoplus_{i=1}^n \bZ \ep_i$ be a lattice of rank $n$.
We denote the set of roots by
$R:=\pr{\pm\ep_i\pm\ep_j \mid i \neq j} \cup 
    \pr{\pm2\ep_i \mid i=1,\dots,n} \subset \frh_\bZ^*$
and denote simple roots by $\alpha_i \in R$ ($i=1,\ldots,n$).
We define the inner product on $\frh_\bZ^*$ by $\br{\ep_i ,\ep_j}:=\delta_{i,j}$,
and define the fundamental weights $\omega_i \in \frh_\bZ^*$ 
by $\br{\omega_i,\alpha_j}=\delta_{i,j}$.
Note that the weight lattice is 
$P := \bZ\omega_1 \oplus \cdots \oplus \bZ\omega_n = \frh_\bZ^*$.
We denote the set of dominant weights by 
$(\frh_\bZ^*)_+:= \pr{\mu\in \frh_\bZ^* \mid \br{\alpha_i^\vee,\mu}\geq0,\ i=1,\dots,n}
 \subset \frh_\bZ^*$.
Here $\alpha_i^\vee$ is the coroot corresponding to $\alpha_i$.
We also denote by $W_0$ the finite Weyl group.

Next we introduce the notations for the affine root system $S$ of type $(C^\vee_n,C_n)$,
and explain the parameters of Koornwinder polynomials.
See \S \ref{sss:r:CvC} and \S \ref{sss:AH:N} for details.
By considering the extension $\wt{\frh}_\bZ^* := \frh_\bZ^* \oplus \bZ \delta$ of 
the lattice $\frh_\bZ^*$, we have the affine root system 
$S:=  \{\pm\ep_i+\frac{k}{2}\delta, \pm2\ep_i+k \delta \mid k \in \bZ,\ i=1,\ldots,n\} 
 \cup \{\pm\ep_i\pm\ep_j+k\delta \mid k \in \bZ, 1 \le i<j\le n\} \subset 
 \wt{\frh}_\bZ^* \otimes \bQ$
of type $(C^\vee_n,C_n)$ and the extended affine Weyl group $W$.
By using the group ring $t(P)$ of the weight lattice $P=\frh_\bZ^*$,
we can present the group $W$ as $W = t(P) \rtimes W_0$.
In the case of rank $n \ge 2$, there are five orbits for the action of $W$ on $S$,
and we consider the parameters associated to these orbits,
denoting them by $(t_0,t,t_n,u_0,u_n)$.
By adding the parameter $q$ and the square root of each parameter,
we define the base field $\bK$ by 
\[
 \bK:=\bQ(q^{\hf},t^{\hf},t_{0}^{\hf},t_{n}^{\hf},u_{0}^{\hf},u_{n}^{\hf}).
\]
Koornwinder polynomials have these five plus one parameters $(q,t_0,t,t_n,u_0,u_n)$.
In the case of rank $n=1$, there are four $W$-orbits,
and the parameters are $(q, t_0,t_n,u_0,u_n)$.
In this case Koornwinder polynomials are equivalent to 
Askey-Wilson polynomials as mentioned before.
By \cite[\S 3]{N} and \cite[(5.2)]{S}, we have the following correspondence to
the original parameters $(q,a,b,c,d)$ of Askey-Wilson polynomials.
\begin{align}\label{eq:AW-K}
 (q,a,b,c,d) = 
 (q,q^\hf t_0^\hf u_0^\hf, -q^\hf t_0^\hf u_0^{-\hf},
  t_n^\hf u_n^\hf, -t_n^\hf u_n^{-\hf}).
\end{align}

For a family $x=(x_1,\ldots,x_n)$ of commutative variables, 
we denote the Laurent polynomial ring of $x_i$'s by $\bK[x^{\pm1}]$.
The finite Weyl group $W_0$ acts on $\bK[x^{\pm1}]$ naturally,
and we denote the invariant ring by $\bK[x^{\pm1}]^{W_0}$.
For a dominant weight $\lambda \in (\frh_\bZ^*)_+$, 
we denote the (monic) Koornwinder polynomials by
\[
 P_\lambda(x) \in \bK[x^{\pm1}]^{W_0}.
\]
We sometimes denote $P_\lambda:=P_\lambda(x)$ for simplicity.
The definition of the Koornwinder polynomial $P_\lambda(x)$ 
will be explained in \S \ref{sss:AH:K}, 
after the review of the affine Hecke algebra $H(W)$ in \S \ref{sss:AH:N} 
and of the double affine Hecke algebra $DH(W)$ in \S \ref{sss:AH:DH}.

\subsection{Littlewood-Richardson coefficients}\label{ss:intro:LR}

The understanding of Macdonald polynomials has been rapidly advanced
since the emergence of the Macdonald-Cherednik theory.
Currently Macdonald polynomials, in particular those of type $A$, 
appear in various fields in mathematics, and have increasing importance. 
However, the study of Koornwinder polynomials seems to be less advanced than 
the Macdonald polynomials of the other root systems, 
and there are many pending problems for the $(C_n^\vee,C_n)$-type.

In this paper, we consider \emph{Littlewood-Richardson coefficients} $c_{\lambda,\mu}^\nu$ 
of Koornwinder polynomials $P_\lambda$, that is the structure constants of the product
in the invariant ring $\bK[x^{\pm1}]^{W_0}$:
\[
 P_\lambda P_\mu = \sum_\nu c_{\lambda,\mu}^\nu P_\nu.
\]
Hereafter we call $c_{\lambda,\mu}^\nu$ \emph{LR coefficients} for simplicity.

Let us recall what is known in the case of type $A$.
The classical LR coefficients are the structure constants of the product 
$s_\lambda s_\mu = \sum_\nu c_{\lambda,\mu}^\nu s_\nu$
of Schur polynomials $s_\lambda$ in the ring of symmetric polynomials. 
We have explicit formulas for the classical LR coefficients via Young tableaux.
Regarding Schur polynomials as the irreducible characters of the general linear group,
we can interpret the coefficient $c_{\lambda,\mu}^\nu$ as 
the multiplicity  of the irreducible decomposition of the tensor representation. 
For Hall-Littlewood polynomials, which are $t$-deformations of Schur polynomials,
we can also consider the LR coefficients $c_{\lambda,\mu}^\nu$, 
and some explicit formulas are known. See \cite[Chap.\ II, (4.11)]{Mb} for example.

Although Macdonald polynomial of type $A$ is a $q$-deformation of Hall-Littlewood polynomial,
no explicit formula for the corresponding LR coefficient $c_{\lambda,\mu}^\nu$ 
had been unknown for a long time.
In \cite[Chap.\ VI, \S 6]{Mb}, Macdonald derived some combinatorial formulas for 
\emph{Pieri coefficients} using arms and legs of Young diagrams.
Here Pieri coefficients mean the LR coefficients $c_{\lambda,\mu}^\nu$
with $\lambda$ the one-row type $(k)$ or the one-column type $(1^l)$,
where the weights are identified with Young diagrams or the partitions.

On the LR coefficients of Macdonald polynomials, Yip \cite{Y} made a great progress. 
Using \emph{alcove walks}, an explicit formula of $c_{\lambda,\mu}^\nu$ is given 
in \cite[Theorem 4.4]{Y} for the Macdonald polynomials of untwisted affine root systems.
Also a simplified formula \cite[Corollary 4.7]{Y} is derived 
in the case $\lambda$ is equal to a minuscule weight.
In particular, this simplified formula recovers Macdonald's formula 
for Pieri coefficients of type $A$ \cite[Theorem 4.9]{Y}.
In Yip's study, the key ingredient is the notion of alcove walk,
originally introduced by Ram \cite{R}.
We will explain the relevant notations and terminology in \S \ref{sss:r:aw}.

\subsection{Main result}

The main result of this paper is the following Theorem \ref{thm:main},
which is a natural $(C^\vee_n,C_n)$-type analogue of 
Yip's alcove walk formulas for LR coefficients in \cite[Theorem 4.4]{Y}.
Let us prepare the necessary notations and terminology for the explanation.

Let $A$ be the fundamental alcove of the extended affine Weyl group $W$ (see \eqref{eq:FA}).
Given an element $w\in W$, we take a reduced expression $w=s_{i_1}\cdots s_{i_r}$.
Given a bit sequence $b=(b_1,\ldots,b_r) \in \{0,1\}^r$ and another element $z \in W$,
we call a sequence of alcoves of the form 
\[
 p = \bigl(p_0:=z A,\ 
           p_1:=z s_{i_1}^{b_1}A,\ 
           p_2:=z s_{i_1}^{b_1}s_{i_2}^{b_2}A,
           \ \ldots, \ p_r:=z s_{i_1}^{b_1} \cdots s_{i_r}^{b_r}A\bigr)
\]
an \emph{alcove walk of type $\oa{w}:=(i_1,\dots,i_r)$ beginning at $z A$}.
We denote by $\Gamma(\oa{w},z)$ the set of such alcove walks.
See Example \ref{eg:ap} for examples of alcove walks.

For an alcove walk $p$, we call the transition $p_{k-1} \to p_k$ \emph{the $k$-th step pf $p$}.
The $k$-th step of $p$ is called a \emph{folding} if $b_k=0$ where 
the bit sequence $b$ corresponds to the alcove walk $p$ (see Table \ref{tab:bp}).

In our main result, we use \emph{a colored alcove walk} introduced by Yip \cite{Y}.
It is an alcove walk equipped with the coloring of folding steps by either black or gray.
We denote by $\Gamma_2^{C}(\oa{w},z)$ the set of colored alcove walks 
whose steps belong to the dominant chamber $C \subset \frh_\bR^* := \frh_\bZ^* \otimes \bR$.

\begin{mth}[{Theorem \ref{thm:LR}}]\label{thm:main}
Let $\lambda, \mu \in \frh_\bZ^*$ be dominant weights, $W_\mu$ be 
the stabilizer of $\mu$ in the finite Weyl group $W_0$ (see \eqref{eq:W_mu}), 
and $W^\mu$ be the complete system of representatives of $W_0/W_\mu$ 
such that the length of each element is shortest in $W_0$ (see \eqref{eq:W^mu}).
Let also $W_\lambda(t)$ be the Poincar\'e polynomial of the stabilizer $W_\lambda$ 
(see \eqref{eq:Wmu(t)}).
We take a reduced expression of the element $w(\lambda) \in W$ in \eqref{eq:w(mu)}.
Then we have
\[
 P_\lambda P_\mu =
 \frac{1}{t_{w_\lambda}^{-\hf} W_\lambda(t)} \sum_{v \in W^\mu} 
 \sum_{p \in \Gamma^C_2(\oa{w(\lambda)}^{-1},(v w(\mu))^{-1})} 
 A_p B_p C_p P_{-w_0.\wtt(p)}.
\]
Here $w_0 \in W_0$ is the longest element, and
the weight $\wtt(p) \in \frh_\bZ^*$ is determined from the element $e(p) \in W$ 
corresponding to the end of the colored alcove walks $p$ as in \eqref{eq:wtdr}.
The coefficients $A_p$, $B_p$ and $C_p$ are factorized, and we have 
\begin{align*}
 A_p := \prod_{\alpha \in w(\mu)^{-1}\cL(v^{-1},v_\mu^{-1})} \rho(\alpha), \quad
 B_p := \prod_{\alpha \in \cL(t(\wgt(p))w_0, e(p))} \rho(-\alpha).
\end{align*}
Here the term $\rho(\alpha)$ is given by 
\begin{align*}
&\rho(\alpha) :=
 \begin{dcases}
  t^\hf 
  \frac{1-t^{-1}q^{\sh(-\alpha)}t^{\hgt(-\alpha)}}
       {1-q^{\sh(-\alpha)}t^{\hgt(-\alpha)}} & (\alpha \not\in W.\alpha_n) \\
  t_n^\hf
  \frac{(1+t_0^{\hf}t_n^{-\hf}q^{\hf\sh(-\alpha)}t^{\hf\hgt(-\alpha)})
        (1-t_0^{-\hf}t_n^{-\hf}
        q^{\hf\sh(-\alpha)} t^{\hf\hgt(-\alpha)})}
       {1-q^{\sh(-\alpha)}t^{\hgt(-\alpha)}} & (\alpha \in W.\alpha_n)
 \end{dcases}, \\
&q^{ \sh(\alpha)}:=q^{-k}, \ 
 t^{\hgt(\alpha)}:=\tprod_{\gamma\in R_+^s}t^{\hf\br{\gamma^\vee,\beta}}
 \tprod_{\gamma\in R_+^{\ell}}(t_0t_n)^{\hf\br{\gamma^\vee,\beta}}
 \quad (\alpha= \beta+k\delta \in S),
\end{align*}
where we used 
$R_+^s:=\pr{\ep_i\pm\ep_j \mid 1 \le i<j \le n}$ and
$R_+^\ell:=\pr{2\ep_i \mid 1 \le i \le n}$.
For the notation $\cL$, see \eqref{eq:cL(v,w)} in \S \ref{sss:r:aw}.
Finally the term $C_p$ is given by $C_p=\prod_{k=1}^r C_{p,k}$ with the factor $C_{p,k}$ 
determined from the $k$-th step of the alcove walk $p$ in Proposition \ref{prp:EPE}.
Here we display the relevant formulas for $C_{p,k}$: 
\begin{align*}
 \psi_i^{\pm}(z) &:=
 \mp\frac{t^{\hf}-t^{-\hf}}{1-z^{\pm1}}
 \qquad (i=1,\ldots,n-1),\\
 \psi_{0}^{\pm}(z) &:=
 \mp\frac{(u_n^\hf-u_n^{-\hf})+z^{\pm\hf}(u_0^\hf-u_0^{-\hf})}{1-z^{\pm1}},\quad
 \psi_{n}^{\pm}(z)  :=
 \mp\frac{(t_n^\hf-t_n^{-\hf})+z^{\pm\hf}(t_0^\hf-t_0^{-\hf})}{1-z^{\pm1}}, \\
 n_i(z) &:=
 \frac{1-t z}{1-z} \frac{1-t^{-1}z}{1-z}
 \hspace{13.7em} (\beta\in W.\alpha_i,\ i=1,\ldots,n-1), \\
 n_0(z) &:=
 \frac{(1-u_n^\hf u_0^\hf z^\hf)(1+u_n^\hf u_0^{-\hf} z^\hf)}{1-z}
 \frac{(1+u_n^{-\hf} u_0^\hf z^\hf)(1-u_n^{-\hf} u_0^{-\hf} z^\hf)}{1-z}
 \quad (\beta\in W.\alpha_0), \\
 n_n(z) &:=
 \frac{(1-t_n^\hf t_0^\hf z^\hf)(1+t_n^\hf t_0^{-\hf} z^\hf)}{1-z}
 \frac{(1+t_n^{-\hf} t_0^\hf z^\hf)(1-t_n^{-\hf} t_0^{-\hf} z^\hf)}{1-z}
 \hspace{2.2em} (\beta\in W.\alpha_n).
\end{align*}
\end{mth}

Note that the term $A_p$ actually depends only on $v \in W^\mu$, 
which corresponds to the beginning of the colored alcove walk $p$.

Let us explain the outline of proof of Theorem \ref{thm:main}.
We denote by $E_\mu(x) \in \bK[X^{\pm1}]$ \emph{the non-symmetric Koornwinder polynomials} 
\cite{Sa,S}, which will be explained in \S \ref{sss:AH:DH}.
Here we need the following two properties.
\begin{itemize}[nosep]
\item
$\{ E_\mu(x) \mid \mu \in \frh_\bZ^*\}$ is a $\bK$-basis of $\bK[x^{\pm1}]$.
\item 
$P_\mu(x)$ is obtained by symmetrizing $E_\mu(x)$  (Fact \ref{fct:Emu}).
More precisely, using the symmetrizer $U$ in \eqref{eq:U}, we have
\[
 P_\mu(x) = \frac{1}{t_{w_\mu}^{-\hf}W_\mu(t)} U E_\mu(x).
\]
\end{itemize}
The outline of proof is a straight $(C_n^\vee, C_n)$-type analogue of 
Yip's derivation in \cite{Y}.
The argument can be divided into four steps, and below we explain them
abbreviating some coefficients and ranges of summations.
\begin{enumerate}[nosep,label=(\roman*)]
\item \label{i:intro:pf1}
For dominant weights $\lambda,\mu \in (\frh_\bZ^*)_+$,
we derive an expansion formula
\[
 x^\mu E_\lambda(x) = \sum_{p \in \Gamma^C} c_p E_{\varpi(p)}(x)
\]
of the product of the non-symmetric Koornwinder polynomial $E_\lambda(x)$ 
and the monomial $x^\mu$ (Corollary \ref{cor:xmuE}).
Here the index set $\Gamma^C$ consists of alcove walks belonging to the dominant chamber $C$.
The symbol $\varpi(p) \in (\frh_\bZ^*)_+$ will be given in \eqref{eq:varpi}.

\item \label{i:intro:pf2}
We use \emph{Ram-Yip type formula} (Fact \ref{fct:RY}),  an expansion formula
for the non-symmetric Koornwinder polynomials in terms of monomials:
\[
 E_\mu(x) = \sum_{p \in \Gamma} f_p t_{\dir(p)}^{\hf}x^{\wgt(p)}.
\]
This formula was derived by Orr and Shimozono  \cite{OS},
based on the work of Ram and Yip \cite{RY} on the same type formula
for the untwisted affine root systems.

\item \label{i:intro:pf3}
Using \ref{i:intro:pf1} and  \ref{i:intro:pf2},
we can calculate the product of the non-symmetric Koornwinder polynomial
$E_\mu(x)$ and the Koornwinder polynomial $P_\lambda(x)$ 
in an extension $\wt{DH}(W)$ of the double affine Hecke algebra $DH(W)$,
and express it as a sum over alcove walks \eqref{eq:EUS1}.
Then we can rewrite it as a sum over colored alcove walks 
and have (Proposition \ref{prp:EPE}):
 \[
 E_\mu(x) P_\lambda(x) = 
 \sum_{v \in W^\lambda} \sum_{p \in \Gamma^C_2} A_p C_p E_{\varpi(p)}(x).
\]

\item \label{i:intro:pf4}
Theorem \ref{thm:main} is obtained by symmetrizing $E_\mu(x)$ in \ref{i:intro:pf3}
and switching $\lambda \leftrightarrow \mu$.
\end{enumerate}


\subsection{Organization and notation}

\subsubsection*{Organization}

We explain the organization of this paper.

In \S \ref{s:K},
we explain Koornwinder polynomials based on the Macdonald-Cherednik theory.
In \S \ref{ss:r}, we explain the root system and alcove walks.
We introduce the root system $R$ of type $C_n$ in \S \ref{sss:r:Cn}, 
the affine root system $S$ of type $(C^\vee_n,C_n)$ in \S \ref{sss:r:CvC}, 
and alcove walks in \S \ref{sss:r:aw}.
In the next \S \ref{ss:AH}, we explain affine Hecke algebras and Koornwinder polynomials.
We introduce the affine Hecke algebra $H(W)$ of type $(C^\vee_n,C_n)$ in \S \ref{sss:AH:N},
and review the basic representation constructed by Noumi \cite{N}.
Then we introduce the double affine Hecke algebra $DH(W)$ of type $(C^\vee_n,C_n)$ 
in \S \ref{sss:AH:DH}, 
and explain the non-symmetric Koornwinder polynomials $E_\lambda$ (Fact \ref{fct:Emu}).
Finally we introduce Koornwinder polynomials $P_\lambda$ in \S \ref{sss:AH:K} (Fact \ref{fct:Pmu}).

In \S \ref{s:LR}, we derive our main Theorem \ref{thm:LR}.
The outline of the discussion is given by the four steps \ref{i:intro:pf1}--\ref{i:intro:pf4} 
previously explained, and the organization of \S \ref{s:LR} follows that.
 
In \S \ref{s:special}, we derive several corollaries of the main Theorem \ref{thm:LR}.
In \S \ref{ss:sp:AW}, we discuss the case of rank $n=1$, 
that is the case of Askey-Wilson polynomials.
In particular, we give a simplified formula for the Pieri coefficient (Proposition \ref{prp:AW_P}),
and recover the recurrence formula of Askey-Wilson polynomials in \cite{AW} from 
our Pieri formula (Remark \ref{rmk:AW2}).
In \S \ref{ss:sp:HL}, we discuss the Hall-Littlewood limit $q \to 0$,
and show that LR coefficients are somewhat simplified (Proposition \ref{prp:HL}).
In \S \ref{ss:rank2} we display examples of LR coefficients in the case of rank $n=2$.

\subsubsection*{Notation and terminology}

Here are the notations and terminology used throughout in this paper. 
\begin{itemize}[nosep]
\item
We denote by $\bZ$ the ring of integers,
by $\bN = \bZ_{\ge 0} := \{0,1,2,\ldots\}$ the set of non-negative integers,
by $\bQ$ the field of rational numbers, and  by $\bR$ the field of real numbers.

\item
We denote by $e$ the unit of a group.
\item
We denote an action of a group $G$ on a set $S$ by 
$g.s$ for $g \in G$ and $s \in S$,
and denote the $G$-orbit of $s$ by $G.s$ or by $G s$.
\item
For a commutative ring $k$ and a family of commutative variants $x=(x_1,x_2,\ldots)$,
we denote by $k[x^{\pm1}]$ the Laurent polynomial ring $k[x_1^{\pm1},x_2^{\pm1},\ldots]$.
\item
We denote by $\delta_{i,j}$ the Kronecker delta.
\end{itemize}

\section{Koornwinder polynomials}\label{s:K}

\subsection{Root systems}\label{ss:r}

\subsubsection{Root systems for type $C_n$}\label{sss:r:Cn}

Let $(R,\frh_{\bZ}^*,R^\vee,\frh_{\bZ})$ be the root data of type $C_n$.
Thus $\frh_\bZ=\bigoplus_{i=1}^n \bZ \ep_i^\vee$ and
$\frh_\bZ^*=\bigoplus_{i=1}^n \bZ \ep_i$ are lattices of rank $n$, 
and we have the non-degenerate bilinear form 
$\br{\ ,\ }:\frh_{\bZ}\times\frh_{\bZ}^* \to \bZ$, $\br{\ep_i^\vee,\ep_j}=\delta_{i,j}$.
We identify $\frh_\bZ^* = \frh_\bZ$ and $\ep_i = \ep_i^\vee$ 
by this bilinear form $\br{\ ,\ }$. 
The set $R$ of roots and the set $R^\vee$ of coroots are given by
\begin{align*}
     R&=\pr{\pm\ep_i\pm\ep_j \mid i \neq j} \cup \pr{\pm2\ep_i \mid i=1,\dots,n}
         \subset \frh_\bZ^*, \\
R^\vee&=\pr{\pm\ep_i\pm\ep_j \mid i \neq j} \cup \pr{\pm \ep_i \mid i=1,\dots,n}
         \subset \frh_\bZ.
\end{align*}
We use the following choice of
the subset $R_+ \subset R$ of positive roots 
and the subset $R^\vee_+ \subset R^\vee$ of positive coroots.
\[
 R_+:=\pr{\ep_i\pm\ep_j \mid i< j}\cup\pr{2\ep_i \mid i=1,\dots,n}, \quad 
 R_+^\vee:=\pr{\ep_i\pm\ep_j \mid i< j}\cup\pr{\ep_i \mid i=1,\dots,n},
\]
We have $R=R_+\sqcup-R_+$ and $R^\vee=R_+^\vee\sqcup -R_+^\vee$.
The simple roots $\alpha_i \in R$ ($i=1,\ldots,n$) are given by
\[
 \alpha_1:=\ep_1-\ep_2,\ \ldots, \ \alpha_{n-1}:=\ep_{n-1}-\ep_n, \ \alpha_n:=2\ep_n.
\]
For each root $\alpha \in R$, we denote the associated coroot by 
$\alpha^\vee := 2\alpha/\br{\alpha,\alpha} \in \frh_\bZ^*=\frh_\bZ$.
The correspondence $\alpha \mapsto \alpha^\vee$ is a bijection, 
and we have $\br{\alpha^\vee,\alpha}=2$.
The coroots for simple roots are
$\alpha^\vee_1=\ep_1-\ep_2$, $\ldots$, $\alpha^\vee_{n-1}=\ep_{n-1}-\ep_n$, 
and $\alpha^\vee_n=\ep_n$. 
We call $\alpha^\vee_i$ simple coroots.

For $\alpha \in R$, we write $s_\alpha$ the reflection by the 
hyperplane $H_{\alpha}:=\pr{x\in \frh_\bR^* \mid \br{\alpha^\vee,x}=0}$
in $\frh^*_\bR=\frh^*_\bZ\otimes_{\bZ}\bR$.
That is, 
\[
 s_\alpha.x := x-\br{\alpha^\vee,x}\alpha,\quad x\in\frh_\bR^*.
\]
We write $s_i:=s_{\alpha_i}$ for $i=1,\dots,n$.
The finite Weyl group $W_0$ is defined to be the subgroup of
$\GL(\frh_\bR^*)$ generated by $s_1,\ldots,s_n$.
As an abstract group, 
$W_0$ is a Coxeter group with generators $s_1,\ldots,s_n$ and relations
\begin{align*}
s_i^2=1 &\quad(i=1,\dots,n),\\
s_is_j=s_js_i &\quad(\abs{i-j}>1),\\
s_is_{i+1}s_i=s_{i+1}s_i s_{i+1} &\quad(i=1,\dots,n-2),\\
s_{n-1}s_ns_{n-1}s_n=s_ns_{n-1}s_ns_{n-1}. &
\end{align*}

Next we introduce notation for weights of the root system of type $C_n$.
For $i=1,\dots,n$, we define $\omega_i:=\ep_1+\dots+\ep_i\in\frh_\bZ^*$, 
and call them the fundamental weights. 
Then we have $\br{\alpha_i^\vee,\omega_j}=\delta_{i,j}$ for $i,j=1,\dots,n$. 
We define the root lattice $Q$ and the weight lattice $P$ by
\begin{align}\label{eq:QhPh*}
 Q := \bZ\alpha_{1}\oplus\cdots\oplus\bZ\alpha_{n} \subset \frh_\bZ^* = 
 P := \bZ\omega_{1}\oplus\cdots\oplus\bZ\omega_{n} \subset \frh_\bR^*.
\end{align}
The action of $W_0 \subset \GL(\frh_\bR^*)$ on $\frh_\bR^*$ 
preserves the weight lattice $P=\frh_\bZ^*$.
We denote this action by
$\lambda \mapsto w.\lambda$ for $w \in W_0$ and $\lambda \in P$.

\subsubsection{Affine root system of type $(C_n^\vee,C_n)$}\label{sss:r:CvC}

Let $t(P)$ be the group algebra of the weight lattice $P=\frh_\bZ^*$.
Denoting by $t(\lambda) \in t(P)$ the element associated to $\lambda \in P$, we have
$t(P)=\pr{t(\lambda) \mid \lambda \in \frh_\bZ^*}$ and 
$t(\lambda)t(\mu)=t(\lambda+\mu)$ $(\lambda,\mu \in \frh_\bZ^*)$. 
Let us consider the lattice extension 
$\wt{\frh}_\bZ^* := \frh_\bZ^*\oplus\bZ\delta$ of $\frh_\bZ^*$ and 
the coefficient extension $\wt{\frh}_\bR^* := \wt{\frh}_\bZ^*\otimes_\bZ\bR$.
We define the action of $t(P)$ on $\wt{\frh}_\bR^*$ by
\[
 t(\lambda).(\mu+m\delta) := \mu+(m-\br{\mu,\lambda})\delta, \quad 
 \mu+m\delta \in \wt{\frh}_\bR^* = \frh_\bR^* \oplus \bR\delta.
\]
The relation of $w \in W_0$ and $t(\lambda)\in t(P)$ in the group $\GL(\wt{\frh}_\bR^*)$
is then given by  $w t(\lambda)w^{-1}=t(w.\lambda)$.
The subgroup $W \subset \GL(\wt{\frh}_\bR^*)$ generated by $t(P)$ and $W_0$
is called \emph{the extended affine Weyl group}. 
That is, 
\begin{align}\label{eq:exafW}
 W := t(P) \rtimes W_0 \subset \GL(\wt{\frh}_\bR^*).
\end{align}

The action of the element $s:=t(\ep_1)s_{2\ep_1} \in W$ on $P=\frh_\bZ^*$ is given by
$s.\ep_1 = \delta-\ep_1$ and $s.\ep_i = \ep_i$ ($i=2,\dots,n$), 
which is the same as the reflection $s_0:=s_{\alpha_0}$
with respect to the hyperplane
$H_{\alpha_0}:=\pr{x\in\frh_\bR^* \mid \br{\alpha_0^\vee,x}=0}$ 
for the affine root $\alpha_0:=\delta-2\ep_1\in\wt{\frh}_\bZ^*$.
Here we set $\alpha_0^\vee := \hf c-\ep_1$,
where $c$ is the basis element in the one-dimensional extension
$\wt{\frh}_\bR:=\frh_\bR \oplus \bR c$ of $\frh_\bR:=\frh_\bZ \otimes_{\bZ} \bR$. 
We also set $\br{c,x}=1$ for all $x\in\frh_\bR^*$.

As an abstract group, $W$ is a Coxeter group with generators $s_0,s_1,\dots,s_n$ and relations
\begin{align*}
 s_i^{2} = 1 & \qquad (i=0,\ldots,n), \\
 s_i s_j = s_j s_i & \qquad (|i-j|>1), \\
 s_i s_{i+1} s_i = s_{i+1} s_i s_{i+1} & \qquad (i=1,\ldots,n-2), \\
 s_i s_{i+1} s_i s_{i+1} = s_{i+1} s_i s_{i+1} s_i & \qquad (i=0,n-1).
\end{align*}
We define the length $\ell(w)$ of an element $w \in W$ to be the length of 
the reduced expression of $w$ by the generators $s_0,\ldots,s_n$.
We also denote by $\bleq$ the corresponding Bruhat order. 
The reduced expressions of $t(\ep_i)$ ($i=1,\dots,n$) are given by
\begin{align}\label{eq:tep}
\begin{split}
 &t(\ep_1) = s_0 s_1 \cdots s_{n-1} s_n s_{n-1} \cdots s_2 s_1, \\
 &t(\ep_2) = s_1 s_0 s_1 \cdots s_n s_{n-1} \cdots s_2, \\
 &t(\ep_i) = s_{i-1} \cdots s_0 s_1 \cdots s_n s_{n-1} \cdots s_i, \\
 &t(\ep_n) = s_{n-1} \cdots s_1 s_0 s_1 \cdots s_n.
\end{split}
\end{align}

Now we define the affine root system $S$ of type $(C_n^\vee, C_n)$ 
in the sense of \cite[(1.3.18)]{M} and\cite{S} by
\begin{align}\label{eq:S}
 S:= \{\pm\ep_i+\frac{k}{2}\delta, \pm2\ep_i+k\delta \mid 
       k\in\bZ,\ i=1,\ldots,n\} \cup 
     \{\pm\ep_i\pm\ep_j+k\delta \mid k \in \bZ, 1\leq i<j\leq n\}
     \subset \wt{\frh}_\bR^*.
\end{align}
We also define the subset $S_+ \subset S$ of positive roots by
\begin{align}\label{eq:S+}
S_+:=\{\alpha+k\delta,\alpha^\vee+\frac{k}{2}\delta \mid 
         \alpha \in R_+,\alpha^\vee \in R_+^\vee,k\in\bN\}\cup
     \{\alpha+k\delta,\alpha^\vee+\frac{k}{2}\delta \mid 
         \alpha\in R_-,\alpha^\vee\in R_-^\vee,k\in\bN\}.
\end{align}
We then have $S=S_+ \sqcup S_-$ with $S_-:=-S_+$. 
We also set $\wt{R}:=R\cup R^\vee$.
Then any $\beta \in S$ can be uniquely written as 
$\beta=\alpha+k\delta\in S$ with $\alpha\in\wt{R}$ and $k \in \hf\bZ$.
We denote the corresponding projection $S \to \wt{R}$ by
\begin{align}\label{eq:StoR}
 \ol{\beta} := \alpha \quad 
 (\beta=\alpha+k\delta \in S, \ \alpha \in \wt{R}, \ k \in \hf\bZ).
\end{align}
We also denote 
\begin{align}\label{eq:wtR_pm}
 \wt{R}_+:=R_+\cup R_+^\vee, \quad \wt{R}_-:=-\wt{R}_+.
\end{align}

Hereafter the case of type $(C_n^\vee,C_n)$ is also called the rank $n$ case.

\subsubsection{Alcove walks}\label{sss:r:aw}

\emph{Alcove walks} are introduced by Ram \cite{R} as analogue of 
Littelmann paths for affine Hecke algebras. 
They are valuable combinatorial objects, and used in Ram-Yip type formula \cite{RY,OS}
for non-symmetric Macdonald-Koornwinder polynomials,
and in Yip's formula \cite{Y} for Littlewood-Richardson rules of Macdonald polynomials 
in the untwisted affine root systems.
In this part we introduce the notation of alcove walks 
which will be used throughout in the text.
Basically we follow the notations in \cite[\S 2.2]{Y}, but make slight modifications.

Let us regard an affine root $\beta=\alpha+k\delta\in S$
($\alpha \in \wt{R}$, $k\in\hf\bZ$) as a affine linear function on $\frh_\bR^*$ by 
\[
 \beta(v) = \br{\alpha,v}+k \quad (v \in \frh_\bR^*).
\]
An alcove is defined to be a connected component of 
the complement 
$ \frh_\bR^* \setminus \bigcup_{\alpha\in S}H_{\alpha}$ 
of the hyperplanes $H_\alpha:=\pr{x\in\frh_\bR^* \mid \alpha(x)=0}$.
The fundamental alcove $A$ is the alcove given by
\begin{equation}\label{eq:FA}
 A := \pr{x\in\frh_\bR^* \mid \alpha_i(x)>0 \ (i=0,\dots,n)}.
\end{equation} 
Its boundary consists of the hyperplanes 
$H_{\alpha_0}, H_{\alpha_1}, \ldots, H_{\alpha_n}$.
Note that the mapping 
\[
 W \ni w \longmapsto w A \in 
 \pi_0(\frh_\bR^* \setminus \tbcup_{\alpha \in S}H_{\alpha})
\]
is a bijection.
An alcove $w A$ is surrounded by $n+1$ hyperplanes, say $H_{\gamma_i}$ ($i=0,\ldots,n$).
We call the intersection $H_{\gamma_i} \cap \ol{w A}$ an \emph{edge} of the alcove $w A$,
where $\ol{w A}$ denotes the closure with respect to the Euclidean topology. 
Note that each hyperplane $H_{\gamma_i}$ separates $w A$ and another alcove $v A$,
which can be written as $v=w s_j$ for some $j=0,\ldots,n$.
Then the edge $H_{\gamma_i} \cap \ol{w A}$ is just the intersection $\ol{w A} \cap \ol{w s_j A}$,
and has two sides, which we call \emph{the $w A$-side} and \emph{the $w s_j A$-side}. 

Given an alcove $w A$, we give a sign $\pm$ to each of the two sides on an edge of $w A$.
Let $H_{\gamma_i}$ ($i=0,\ldots,n$) be the hyperplanes surrounding $w A$.
By renaming the indices $i$ if necessary,
we can assume that the hyperplane $H_{\gamma_i}$ separates $w A$ and $w s_i A$.
Then using the projection $\gamma_i \mapsto \ol{\gamma_i}$ in \eqref{eq:StoR} 
and the symbols $\wt{R}_{\pm}$ in \eqref{eq:wtR_pm}, 
we set the signs by the following rule.
\begin{itemize}[nosep]
\item If $\ol{\gamma_i} \in \wt{R}_+$, 
then the $w A$-side of $H_{\gamma_i} \cap \ol{w A}$ is assigned by $+$ 
and the $w s_i A$-side is by $-$.
\item If $\ol{\gamma_i} \in \wt{R}_-$, 
then $w A$-side is assigned by $-$ and the $w s_i A$-side is by $+$.
\end{itemize}
See Figure \ref{fig:as} for the assignment in the rank $2$ case.

\begin{figure}[htbp]
\centering
\includegraphics[bb= 0 0 176 173]{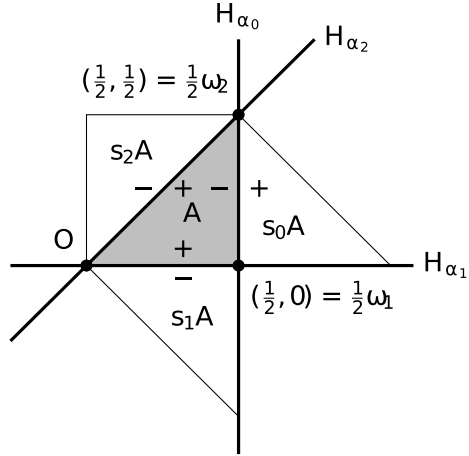}
\caption{Signs for the edges of the fundamental alcove $A$ in the rank $2$ case} \label{fig:as}
\end{figure}

Given an element $w \in W$ and a reduced expression $w=s_{i_1}\cdots s_{i_r}$, 
we define a subset $\cL(w) \subset S$ by 
\begin{align}\label{eq:cL(w)}
 \cL(w):=\pr{\alpha_{i_1}, \, s_{i_1}\alpha_{i_2},  \, \dots,  \, 
             s_{i_1} \cdots s_{i_{r-1}}\alpha_{i_r}}.
\end{align}
The set $\pr{H_{\beta} \mid \beta \in \cL(w)}$ consists of the
hyperplanes separating $A$ and $wA$.
Given elements $v,w \in W$ and their reduced expressions, we also set 
\begin{align}\label{eq:cL(v,w)}
 \cL(v,w):=(\cL(v)\cup\cL(w)) \setminus (\cL(v)\cap\cL(w)).
\end{align}
The set $\pr{H_\beta \mid \beta\in\cL(v,w)}$ consists of the
hyperplanes separating $v A$ and $w A$.
If $v \bleq w$, then we have
\begin{align}\label{eq:cL:trs}
 \cL(v,w) = v.\cL(e,v^{-1}w) = v.\cL(v^{-1}w).
\end{align}

Let us again given $w \in W$ and a reduced expression $w=s_{i_1}\cdots s_{i_r}$.
Then the mapping 
\begin{align*}
 \{0,1\}^r \ni (b_1,\dots,b_r) \longmapsto 
 s_{i_1}^{b_1} \cdots s_{i_r}^{b_r} \in \pr{v\in W \mid v\bleq w}
\end{align*}
is a bijection.
Let us given extra $z,w \in W$ such that $v \bleq w$.
We can write $v=s_{i_1}^{b_1}\cdots s_{i_r}^{b_r}$ with $b=(b_1,\dots,b_r) \in \{0,1\}^r$. 
We then consider the following sequence $p$ of alcoves.
\[
 p = \bigl(p_0:=z A,\ 
           p_1:=z s_{i_1}^{b_1}A,\ 
           p_2:=z s_{i_1}^{b_1}s_{i_2}^{b_2}A,
           \ \ldots, \ p_r:=z s_{i_1}^{b_1} \cdots s_{i_r}^{b_r}A\bigr).
\]
The sequence $p$ is called an \emph{alcove walk of type
$\oa{w}=(i_1,\ldots,i_r)$ beginning at $z A$},
and we denote by $\Gamma(\oa{w},z)$ the set of alcove walks of this kind.
The symbol $\oa{w}$ emphasizes that we choose a reduced expression $w=s_{i_1}\cdots s_{i_r}$.

\begin{eg}[Alcove walks in the rank $2$ case]\label{eg:ap}
For $w=s_1s_2 s_1 s_0$ and $z=e\in W$, the two alcove walks 
\[
 p_1 := (A, A, s_2A, s_2s_1A, s_2s_1s_0A), \ 
 p_2 := (A, s_1A, s_1s_2A, s_1s_2s_1A, s_1s_2s_1s_0A) \in \Gamma(\oa{w},z)
\]
are shown in Figure \ref{fig:ap}, where the gray region
is the fundamental alcove $A$, and the number $i=0,1,2$ 
on a hyperplane means that it belongs to the $W$-orbit of $H_{\alpha_i}$.

\begin{figure}[htbp]
\centering
\includegraphics[bb=0 0 324 153]{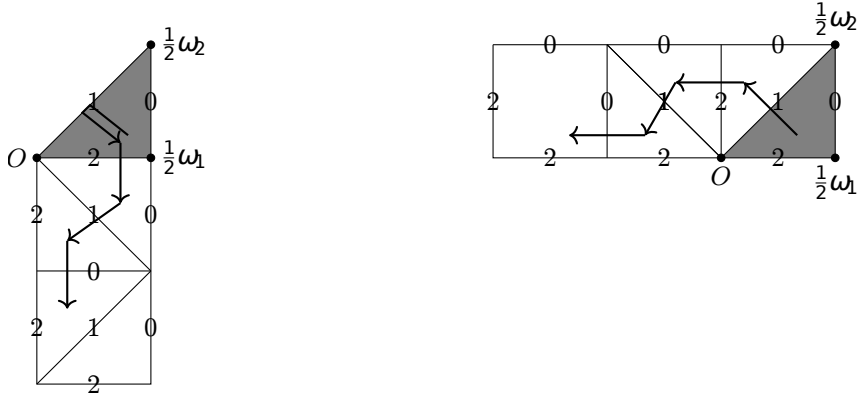}
\caption{Alcove walks $p_1$ and $p_2$}
\label{fig:ap}
\end{figure}
\end{eg}

For an alcove walk $p \in \Gamma(\oa{w},z)$ and $k=1,\ldots,r$, 
the transition $p_{k-1} \to p_k$ is called \emph{the $k$-th step of $p$}.
The $k$-th step is called a \emph{crossing} if $b_k=1$, 
and called a \emph{folding} if $b_k=0$.
The correspondence between the bit $b_k$ and the $k$-th step is shown in Table \ref{tab:bp},
where we denote by $v_{k-1} \in W$ the element such that $p_{k-1} = v_{k-1}A$.

\begin{table}[htbp]
\centering
\includegraphics[bb= 0 0 231 80]{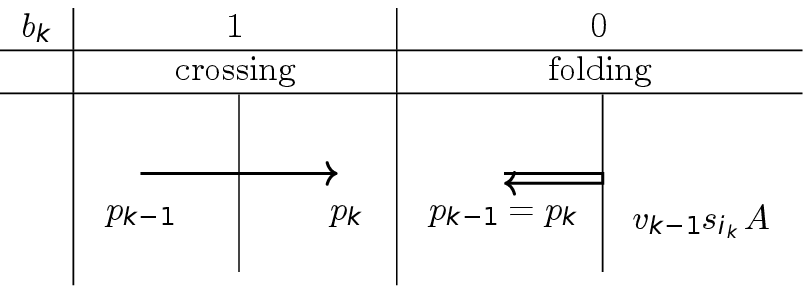}
\caption{Correspondence between bits and steps}\label{tab:bp}
\end{table}

Let us again given $z,w \in W$ with a reduced expression $w=s_{i_1}\cdots s_{i_r}$. 
For an alcove walk 
$p = (z A, \ldots, z s_{i_1}^{b_1} \cdots s_{i_r}^{b_r}A) \in \Gamma(\oa{w},z)$, 
we define $e(p) \in W$ by
\begin{align}\label{eq:e(p)}
 e(p) := z s_{i_1}^{b_1} \cdots s_{i_r}^{b_r}.
\end{align}
Thus $e(p)$ corresponds to the end of $p$.
We also define $h_k(p) \in S$ for $k=1,\ldots,r$ by the following rule.
Denote $v := s_{i_1}^{b_1} \cdots s_{i_{k-1}}^{b_{k-1}}$ for simplicity,
so that we have $p_{k-1}=v A$.
Then we define 
\begin{align}\label{eq:hk(p)}
h_k(p) :=  
\text{the affine root such that the corresponding hyperplane $H_{h_k(p)}$ 
      separates $v A$ and $v s_{i_k} A$}.
\end{align}
Furthermore, we call the $k$-th step of $p \in \Gamma(\oa{w},z)$ an \emph{ascent} if 
$z s_{i_1}^{b_1} \cdots s_{i_{k-1}}^{b_{k-1}} \bleq 
 z s_{i_1}^{b_1} \cdots s_{i_{k}}^{b_k}$, 
and call it a \emph{descent} if 
$z s_{i_1}^{b_1}\cdots s_{i_{k-1}}^{b_{k-1}} \bgeq 
 z s_{i_1}^{b_1} \cdots s_{i_{k}}^{b_k}$.
We denote the set of descent steps of $p$ by
\begin{align}\label{eq:des/asc}
 \des(p) := \pr{k=1,\ldots,r \mid \text{the $k$-th step is a descent}}.
\end{align}

Recalling the sign on an edge of an alcove (see Figure \ref{fig:as} for an example),
we can classify each step of an alcove walk $p$ into four types as in Table \ref{tab:sp},
where we used the symbol $v_{k-1} \in W$ such that $p_{k-1} = v_{k-1} A$.

Using this classification, we define $\varphi_{\pm}(p) \subset \{1,\ldots,r\}$ by
\begin{align}\label{eq:vp(p)}
\begin{split}
 \varphi_+(p) &:= \pr{k \mid \text{the $k$-th step of $p$ is a positive folding}},\\
 \varphi_-(p) &:= \pr{k \mid \text{the $k$-th step of $p$ is a negative folding}},
\end{split}
\end{align}
and define $\xi_{\des}(p) \subset \{1,\ldots,r\}$ by 
\begin{align}\label{eq:xp(p)}
  \xi_{\des}(p)&:=\pr{k\mid \text{the $k$-th step of $p$ is a crossing and $k\in\des(p)$}}.
\end{align}
Note that we fix a reduced expression $w=s_{i_1}\cdots s_{i_r}$
in the definitions of $\varphi_{\pm}(p)$ and $\xi_{\des}(p)$.

\begin{table}[htbp]
\centering
\includegraphics[bb= 0 0 414 68]{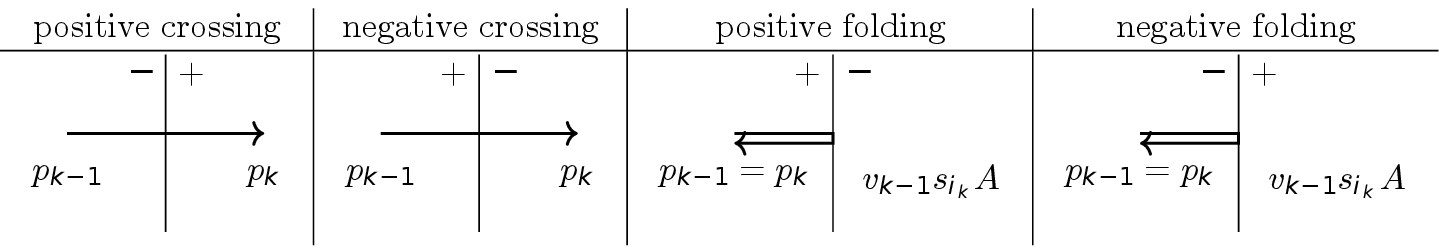}
\caption{Classification of steps in alcove walks}
\label{tab:sp}
\end{table}

\subsection{Affine Hecke algebras and Koornwinder polynomials}\label{ss:AH}

In this subsection, 
we explain the realization of non-symmetric Koornwinder polynomials 
via the polynomial representation of the 
affine Hecke algebra type $C_n$, 
and introduce Koornwinder polynomials by their symmetrization. 

\subsubsection{Affine Hecke algebras of type $(C_n^\vee,C_n)$ and polynomial representations}
\label{sss:AH:N}

Recall the affine root system $S$ of type $(C_n^\vee,C_n)$ 
and the extended affine Weyl group $W$ explained in \S \ref{sss:r:CvC}.
Let $\pr{t_\alpha \mid \alpha \in S}$ be parameters satisfying the condition
$t_\alpha = t_\beta$ for $\beta \in W.\alpha$.
Since the $W$-orbits in $S$ are given by
\begin{align*}
 W.\alpha_i=W.\alpha_i^\vee \ (i=1,\ldots,n-1),\ 
 W.\alpha_n,\ W.\alpha_n^\vee,\ W.\alpha_0,\ W.\alpha_0^\vee,
\end{align*}
we can replace the family $\{t_\alpha\}$ by
\begin{align}\label{eq:t_alpha}
 (t_{\alpha_0},t_{\alpha_i}=t_{\alpha_i^\vee},t_{\alpha_n},t_{\alpha_0^\vee},t_{\alpha_n^\vee})
=(t_0,t,t_n,u_0,u_n).
\end{align}
We will also denote $t_1,\dots,t_{n-1}:=t$.
Now we set the base field $\bK$ as
\begin{align}\label{eq:bK}
 \bK := \bQ(q^{\hf},t^{\hf},t_{0}^{\hf},t_{n}^{\hf},u_{0}^{\hf},u_{n}^{\hf}),
\end{align}
and all the linear spaces, their tensor products, 
and the algebras will be those over $\bK$ unless otherwise stated.

The affine Hecke algebra $H(W)$ is the associative algebra generated by
$T_0, T_1, \ldots, T_n$ subject to the following relations.
\begin{align}
 \nonumber
 (T_i-t_i^\hf)(T_i+t_i^{-\hf})=0 & \qquad(i=0,\ldots,n),\\
 \label{eq:br1}
 T_i T_j = T_j T_i & \qquad (|i-j|>1, (i,j) \not\in \{(n,0),(0,n)\}), \\
 T_i T_{i+1} T_i = T_{i+1} T_i T_{i+1} & \qquad (i=1,\ldots,n-2),\\ 
 \label{eq:br3}
 T_i T_{i+1} T_i T_{i+1} = T_{i+1} T_i T_{i+1} T_i & \qquad (i=0,n-1).
\end{align}
The relations \eqref{eq:br1}--\eqref{eq:br3} are called \emph{the braid relations}. 

Given an element $w \in W$ together with a reduced expression $w=s_{i_1}\cdots s_{i_r}$,
we consider the alcove walk
$(A, s_{i_1}A, \ldots, s_{i_1} \cdots s_{i_r}A=wA) \in \Gamma(\oa{w},e)$, 
and define $Y^w \in H(W)$ by
\begin{align}\label{Y:rel}
 Y^w := T_{i_1}^{\ep_1}\cdots T_{i_{r}}^{\ep_r},
\end{align}
where we set $\ep_k := 1$ if the $k$-th step of $p$ is a positive crossing, 
and set $\ep_k :=-1 $ if the $k$-th step is a negative crossing 
according to the classification in Figure \ref{tab:sp}.
The decomposition of $Y^w$ by $T_i$'s is independent of the choice of a reduced expression of $w$.
By the relations of $H(W)$, 
we find that the family $\{Y^w \mid w \in W\}$ is mutually commutative \cite[\S2]{N}. 

As explained in \cite[\S3]{M}, we can calculate $Y^{t(\ep_i)}$ 
using the reduced expression of $t(\ep_i)$ in \eqref{eq:tep}.
The result is 
\begin{align}\label{eq:Yred}
\begin{split}
 Y^{t(\ep_1)} &= T_{0}\cdots T_n T_{n-1}\cdots T_1, \\
 Y^{t(\ep_2)} &= T_{1}^{-1} T_0 \cdots T_{n-1} T_n T_{n-1}\cdots T_2, \\
 Y^{t(\ep_i)} &= T_{i-1}^{-1} \cdots T_1^{-1} T_0 \cdots T_{n-1} T_n T_{n-1}\cdots T_i, \\
 Y^{t(\ep_n)} &= T_{n-1}^{-1} \cdots T_1^{-1} T_0 T_1 \cdots T_n.
\end{split}
\end{align}

Now we denote by
\[
 \bK[Y^{\pm1}] = \bK[Y_1^{\pm1},\ldots,Y_n^{\pm1}] \subset H(W), \quad 
 Y_i := Y^{t(\ep_i)} \quad (i=1,\dots,n)
\]
the ring of Laurent polynomials in $Y_1,\dots,Y_n$. 
Then we have an isomorphism $H(W)\simeq H(W_0)\otimes \bK[Y^{\pm1}]$,
where $H(W_0)$ is the Hecke algebra of the finite Weyl group $W_0$.
The latter is the subalgebra of $H(W)$ generated by $T_1,\ldots,T_n$.

Next we review the basic representation of the affine Hecke algebra $H(W)$ introduced by Noumi \cite{N}. 
Let $\bK(x)=\bK(x_1,\ldots,x_n)$ be the field of rational functions with $n$ variables. 
Then the mapping 
\begin{align}\label{eq:NRep}
\begin{split}
 T_i& \longmapsto t_i^{\hf}+t_i^{-\hf}
 \frac{1-t_ix_i/x_{i+1}}{1-x_i/x_{i+1}}(s_i-1) \quad (i=1,\ldots,n-1),\\
 T_{0}& \longmapsto t_{0}^{\hf}+t_{0}^{-\hf}
 \frac{(1-u_0^{\hf}t_0^{\hf} q^{\hf}x_1^{-1})(1+u_0^{-\hf}t_0^{\hf}q^{\hf}x_1^{-1})}
      {1-q x_1^{-2}}(s_0-1), \\
 T_n& \longmapsto t_n^{\hf}+t_n^{-\hf}
 \frac{(1-u_n^{\hf}t_n^{\hf}x_n)(1+u_n^{-\hf}t_n^{\hf}x_n)}{1-x_n^2}(s_n-1)
\end{split}
\end{align}
defines a ring homomorphism $\rho: H(W) \to \End(\bK(x))$.
Moreover its image is contained in the endomorphism algebra 
$\End_{\bK}(\bK[x^{\pm1}]) \subset \End_{\bK}(\bK(x))$ of the Laurent polynomials. 
We call $\rho$ \emph{the basic representation of $H(W)$}.
Hereafter we identify $H(W)$ and its image under $\rho$, 
and regard $H(W)$ as a subalgebra of $\End_{\bK}(\bK[x^{\pm1}])$.
The right hand sides of \eqref{eq:NRep} are $q$-difference operators 
called \emph{Dunkl operators of type $(C^\vee_n,C_n)$}.

Let us give a simplified description of \eqref{eq:NRep}.
Using
\[
 u_i := \begin{cases}1 & (i=1,\ldots,n-1) \\ u_0 & (i=0)\\ u_n & (i=n)\end{cases}, \quad
 x^{\alpha_i}:=
 \begin{cases}x_i/x_{i+1} &(i=1,\ldots,n-1) \\ q x_1^{-2} & (i=0)\\ x_n^2 & (i=n)\end{cases},
\]
we can rewrite $T_i$'s as 
\begin{align}\label{eq:NRcpt}
 T_i = t_i^{\hf}+t_i^{-\hf}
 \frac{(1-u_i^{\hf}t_i^{\hf}x^{\frac{\alpha_i}{2}})
      (1+u_i^{-\hf}t_i^{\hf}x^{\frac{\alpha_i}{2}})}
      {1-x^{\alpha_i}}(s_i-1),
\end{align}
where we identified the left and right hand sides in \eqref{eq:NRep} as claimed before.
Let us further define the rational functions
$c_i(z), d_i(z) \in \bK(z)$ by
\begin{align}\label{eq:cd}
 c_i(z):=t_i^{-\hf}
 \frac{(1-u_i^{\hf}t_i^{\hf}z^{\hf})
      (1+u_i^{-\hf}t_i^{\hf}z^\hf)}{1-z}, \quad 
 d_i(z):= t_i^{\hf}-c_i(z)
 = \frac{(t_i^\hf-t_i^{-\hf})+(u_i^\hf-u_i^{-\hf})z^\hf}{1-z}.
\end{align}
Then we can rewrite \eqref{eq:NRep} or \eqref{eq:NRcpt} as 
\begin{align}\label{eq:Ticd}
 T_i= t_i^{\hf}+c_i(x^{\alpha_i})(s_i-1)
 = t_i^{\hf}s_i+d_i(x^{\alpha_i})(1-s_i)
 = c_i(x^{\alpha_i})s_i+d_i(x^{\alpha_i}).
\end{align}

For later use, 
we calculate the action of the element $Y^\beta$ on $1$ in the basic representation
for an affine root $\beta = \alpha+k\delta \in S$ ($\alpha \in \wt{R}$, $k \in \bZ$).
Let us define
\begin{align}\label{eq:shhgt}
 q^{ \sh(\alpha+k\delta)}:=q^{-k}, \quad
 t^{\hgt(\alpha+k\delta)}:=\tprod_{\gamma \in R_+^s}t^{\hf\br{\gamma^\vee,\alpha}}
 \tprod_{\gamma \in R_+^{\ell}}(t_0t_n)^{\hf\br{\gamma^\vee,\alpha}}.
\end{align}
Here $R_+^s:=\pr{\ep_i\pm\ep_j \mid 1 \le i<j \le n}$ denotes the set of positive short roots,
and $R_+^\ell:=\pr{2\ep_i \mid 1 \le i \le n}$ denotes the set of of positive long roots.
Then we can check 
\begin{align}\label{eq:Ylam1}
 Y^\beta 1 = q^{\sh(\beta)}t^{\hgt(\beta)}.
\end{align}
See also \cite[Proposition 4.5]{S} for a more general formula.

Finally we recall \emph{the Lusztig relations} in the basic representations of 
affine Hecke algebra.
For each weight $\lambda=(\lambda_1,\dots,\lambda_n)\in P=\frh_\bZ^*$, 
we define $x^\lambda \in \bK[x^{\pm1}]$ by
\begin{align}\label{eq:xlam}
 x^\lambda:=x_1^{\lambda_1}\cdots x_n^{\lambda_n} \in \bK[x^{\pm1}].
\end{align}

\begin{fct}[{Lusztig relations, \cite[Proposition 3.6]{L}}]\label{fct:Lus}
For $i=0,\dots,n$ and $\lambda \in \frh_\bZ^*$, we have
\begin{align*}
 T_i x^{\lambda}-x^{s_i.\lambda}T_i=
 d_i(x^{\alpha_i})(x^\lambda-x^{s_i.\lambda}),
\end{align*}
where the rational function $d_i(z)$ is defined by \eqref{eq:cd}.
\end{fct}

\subsubsection{Double affine Hecke algebras and non-symmetric Koornwinder polynomials}
\label{sss:AH:DH}

Next we review \emph{the double affine Hecke algebra $DH(W)$ of type $(C_n^\vee,C_n)$}
and \emph{the non-symmetric Koornwinder polynomials $E_\lambda(x)$}, 
following \cite{M}, \cite{Sa} and \cite{S}.

As in the previous \S \ref{sss:AH:N}, we regard $H(W)$ as a $\bK$-subalgebra 
of $\End_{\bK}(\bK[x^{\pm1}])$ by the basic representation \eqref{eq:NRep}.
We define \emph{the double affine Hecke algebra} $DH(W) \subset \End_{\bK}(\bK[x^{\pm1}])$
as the $\bK$-subalgebra generated by $\bK[x^{\pm1}]$, $H(W_{0})$ and $\bK[Y^{\pm1}]$.
Thus 
\[
 DH(W) := \bigl< \bK[x^{\pm1}], \, H(W_{0}), \, \bK[Y^{\pm1}] \bigr>
 \subset \End_{\bK}(\bK[x^{\pm1}]).
\] 
As in the case of untwisted affine root systems,
the algebra $DH(W)$ has \emph{the Cherednik anti-involution} $\phi$ \cite[\S 3]{Sa}:  
\begin{align}\label{eq:inv}
\begin{split}
 &\phi(x_i)=Y_i^{-1}, \quad
  \phi(Y_i)=x_i^{-1}, \quad
  \phi(T_i)=T_i\quad(i=1,\dots,n),\\
 &\phi(u_{n})=t_{0},\quad\phi(t_{0})=u_{n}.
 \end{split}
\end{align}
On the element $T_0$ the anti-involution acts as $\phi(T_0)=T_{s_{2\ep_1}}^{-1}x_1^{-1}$.
In fact, we have $T_0=Y_1T_{s_{2\ep_1}}^{-1}$ and 
$T_{s_{2\ep_1}}=T_1\cdots T_{n}T_{n-1}\cdots T_{1}$ by \eqref{eq:Yred}.
Hereafter we denote 
\begin{align}\label{eq:Tiv}
 T_i^\vee := \phi(T_i) \quad (i=0,\dots,n).
\end{align}

Next we introduce \emph{the $x$- and $Y$-intertwiners} for $DH(W)$ following \cite[\S 5.6]{M}.  
Let $\wt{DH}(W)$ be the coefficient extension of $DH(W)$ by rational functions of $x$'s and $Y$'s.
In other words, we set 
\begin{align}\label{eq:wDH}
 \wt{DH}(W) := \bigl< \bK(x), \, H(W_0), \, \bK(Y) \bigr> 
 \subset \End_\bK(\bK(x)).
\end{align}
Here $\bK(x)$ and $\bK(Y)$ are the fields of rational functions of $x_i$ and $Y_i$ 
$(i=1,\dots,n)$ respectively. 
For $i = 0,\ldots,n$, we define $S_i^{x} \in \wt{DH}(W)$ by
\begin{align}\label{eq:Sxi}
 S^x_i := T_i+\varphi_i^{+}(x^{\alpha_i})
        = T_i^{-1}+\varphi^{-}_i(x^{\alpha_i}),
\end{align}
where 
\begin{align}\label{eq:varphi}
 \varphi_i^{\pm}(z) :=
 \mp \frac{(t_i^\hf-t_i^{-\hf})+z^{\pm\hf}(u_i^\hf-u_i^{-\hf})}{1-z^{\pm1}}
 \in \bK(z).
\end{align}
We call $S_i^x$ \emph{the $x$-intertwiners}.

Let us explain some basic properties of $x$-intertwiners.
Recalling the rational function $d_i(z)$ in \eqref{eq:cd} 
and the expression of $T_i$ in \eqref{eq:Ticd}, we have 
\begin{align}\label{eq:Sics}
 \varphi_i^+(z) = d_i(z), \quad 
 S_i^x= T_i - d_i(x^{\alpha_i}) = c_i(x^{\alpha_i}) s_i.
\end{align}
For each weight $\lambda \in \frh_\bZ^*$, we have 
\begin{align}\label{eq:Sxl}
 S_i^x x^{\lambda} = x^{s_i(\lambda)} S_i^x
\end{align}
by the Lusztig relations (Fact \ref{fct:Lus}).
Moreover, by \cite[(5.5.2)]{M}, the $x$-intertwiners $S_i^x$ ($i=0,\ldots,n$) 
satisfy the same braid relations as \eqref{eq:br1}--\eqref{eq:br3}: 
\begin{align}\label{eq:brSx}
\begin{split}
 S^x_i S^x_j = S^x_j S^x_i & \quad (|i-j|>1),\\
 S^x_i S^x_{i+1} S^x_i = S^x_{i+1} S^x_i S^x_{i+1} & \quad  (i=1,\dots,n-2),\\
 S^x_i S^x_{i+1} S^x_i S^x_{i+1} = S^x_{i+1} S^x_i S^x_{i+1} S^x_i & \quad  (i=0,n-1).
\end{split}
\end{align}
Given an element $w \in W$, choose 
a reduced expression $w=s_{i_1} \cdots s_{i_p}$, and set 
\begin{align}\label{eq:Sxw}
 S_w^x := S_{i_1}^x \cdots S_{i_p}^x \in \wt{DH}(W). 
\end{align}
By the braid relations, 
$S_w^x$ is independent of the choice of a reduced expression of $w$.

Next we introduce \emph{$Y$-intertwiners}.
First, note that the anti-involution $\phi$ can  be extended to $\wt{DH}(W)$.
In fact, $\wt{DH}(W)$ is the Ore localization of the non-commutative algebra $DH(W)$ 
by the commutative subalgebras $\bK[x^{\pm1}]$ and $\bK[Y^{\pm1}]$,
and $\phi$ is an isomorphism on these commutative subalgebras.
We denote the extension of $\phi$ to $\wt{DH}(W)$ by same symbol $\phi$.
Now we define the $Y$-intertwiners $S_i^Y \in \wt{DH}(W)$ by
\begin{align}\label{eq:SYi}
\begin{split}
 S_i^{Y} &:= \phi(S_i^x) = T_i+\psi_i^{+}(Y^{-\alpha_i})
             = T_i^{-1}+\psi_i^{-}(Y^{-\alpha_i}) 
 \quad (i=1,\ldots,n),\\
 S_{0}^{Y} &:= \phi(S_0^x) = T_{0}^\vee+\psi_{0}^{+}(qY_{1}^{2})
             = (T_{0}^{\vee})^{-1}+\psi^-_{0}(qY_{1}^{2}),
\end{split}
\end{align}
where the symbols $\psi_i^{\pm}(z)$ denote the images of $\varphi_i^{\pm}(z)$ 
given in \eqref{eq:varphi} under the extended anti-involution $\phi$. 
That is, we have 
\begin{align}\label{eq:psi(z)}
\begin{split}
 \psi_i^{\pm}(z)
 &:=\varphi_i^{\pm1}(z)
 =\mp\frac{t^{\hf}-t^{-\hf}}{1-z^{\pm1}}
 \quad(i=1,\ldots,n-1),\\
 \psi_{0}^{\pm}(z)
 &:=\mp\frac{(u_n^\hf-u_n^{-\hf})+z^{\pm\hf}(u_0^\hf-u_0^{-\hf})}{1-z^{\pm1}},\\
 \psi_{n}^{\pm}(z)
 &:=\mp\frac{(t_n^\hf-t_n^{-\hf})+z^{\pm\hf}(t_0^\hf-t_0^{-\hf})}{1-z^{\pm1}}.
\end{split}
\end{align}

We can deduce properties of $S_i^Y$'s from those of $S_i^x$'s.
For example, applying the anti-involution $\phi$ to the relation \eqref{eq:Sxl}, we have
\begin{align}\label{eq:SYl}
 S_i^Y Y^\lambda = Y^{s_i\lambda}S_i^Y
\end{align}
for each $i=0,\ldots,n$ and $\lambda \in \frh_\bZ^*$.
We can also see that $S_i^Y$'s satisfy the same braid relations as \eqref{eq:brSx}.
For an element $w \in W$, we can define $S_w^Y \in \wt{DH}(W)$ 
by choosing a reduced expression $w=s_{i_1} \cdots s_{i_p}$ and 
\begin{align}\label{eq:SYw}
 S_w^Y := S_{i_1}^Y \cdots S_{i_p}^Y \in \wt{DH}(W).
\end{align}
It is well-defined by the braid relations of $S_i^Y$'s.
 
Finally we explain the non-symmetric Koornwinder polynomials.
For each weight $\mu \in P = \frh_\bZ^*$, we regard $t(\mu) W_0 \subset W$ 
by the decomposition $W = t(P) \rtimes W_0$ in \eqref{eq:exafW}.
Then we define $w(\mu) \in W$ by the following description:
\begin{align}\label{eq:w(mu)}
 \text{$w(\mu)$ is the shortest element among $t(\mu) W_0 \subset W$.}
\end{align}
Now we have:

\begin{fct}[{\cite[\S 6]{Sa}, \cite[Theorem 4.8]{S}}]\label{fct:Emu}
For $\mu \in \frh_\bZ^*$, the element
\[
 E_\mu(x) := S_{w(\mu)}^Y 1
\]
belongs to $\bK[x^{\pm1}]$.
We call it 
\emph{the non-symmetric Koornwinder polynomial associated to $\mu$}.
\end{fct}

By \eqref{eq:SYl}, $E_\mu(x)$ is a simultaneous eigenfunction 
of the family $\pr{Y^\lambda \mid \lambda \in \frh_\bZ^*}$ of Dunkl operators.
Note that our normalization of $E_\mu(x)$ is different from that in \cite{Sa,S}.
In loc.\ cit., the coefficient of $x^\mu$ is normalized to be $1$.

\subsubsection{Koornwinder polynomials}\label{sss:AH:K}

Now we introduce \emph{Koornwinder polynomials} by 
symmetrizing non-symmetric Koornwinder polynomials.

First, we define the set $(\frh_\bZ^*)_+ \subset \frh_\bZ^*$ of dominant weights by
\[
 (\frh_\bZ^*)_+:=
 \pr{\mu\in \frh_\bZ^* \mid \br{\alpha_i^\vee,\mu}\geq0,\ i=1,\dots,n}.
\]
For a dominant weight $\mu\in(\frh_\bZ^*)_+$, 
we denote the stabilizer of $\mu$ in the finite Weyl group $W_0$ by
\begin{align}\label{eq:W_mu}
 W_\mu := \{w \in W_0 \mid w.\mu=\mu\} \subset W_0,
\end{align}
and denote the longest element among $W_\mu$ by
\begin{align}\label{eq:wmu}
 w_\mu \in W_\mu.
\end{align}

Next, using the notations in \S \ref{sss:r:CvC} and \S \ref{sss:AH:N}, 
we define $t_w \in \bK$ for each $w \in W$ by 
\begin{align}\label{eq:t_w}
 t_w := \prod_{\beta\in\cL(w)}t_\beta \in \bK.
\end{align}
Here $\{t_\alpha \mid \alpha \in S\}$ is the $W$-invariant family of 
parameters \eqref{eq:t_alpha}, $\bK$ is the base field \eqref{eq:bK},
and $\cL(w) \subset S$ is given by \eqref{eq:cL(w)}.
If $w=s_{i_1}\cdots s_{i_r}\in W$ is the shortest element, then we have
$t_w=t_{i_1}\cdots t_{i_r}$.
For a dominant weight $\mu\in(\frh_\bZ^*)_+$, we define \emph{the Poincar\'e polynomial}
$W_\mu(t) \in \bK$ of the stabilizer $W_\mu$ by 
\begin{align}\label{eq:Wmu(t)}
 W_\mu(t):=\sum_{u \in W_\mu}t_u.
\end{align}

\begin{lem}
For each element $\mu\in(\frh_\bZ^*)_+$, we have
\[
 \sum_{u\in W_\mu}
 \biggl(\prod_{\alpha\in \cL(1,u)} t_\alpha^\hf
        \frac{1-t^{\hgt(-\alpha)}t_\alpha^{-1}}{1-t^{\hgt(-\alpha)}}\biggr)
 \biggl(\prod_{\alpha\in \cL(u,w_\mu)} t_\alpha^{-\hf}
        \frac{1-t^{\hgt(-\alpha)}t_\alpha}{1-t^{\hgt(-\alpha)}}\biggr)
 =t_{w_\mu}^{-\hf}W_\mu(t).
\]
\end{lem}

For a proof, see \cite[Lemma 3.4]{Y}.

Next we define the \emph{symmetrizer} $U$ by 
\begin{align}\label{eq:U}
 U:=\sum_{w\in W_0}t_{w_0 w}^{-\hf}T_w.
\end{align}
By \cite[(5.5.9)]{M}, we then have
\begin{align}\label{eq:UTUt}
 U T_i=U t_i^\hf, \quad T_i U =t_i^\hf U \quad (i=1,\dots,n).
\end{align}

Hereafter we denote the ring of $W_0$-invariant Laurent polynomials by 
\[
\bK[x^{\pm1}]^{W_0}:=\pr{f\in\bK[x^{\pm1}] \mid w.f=f, \, w\in W_0}.
\]
Here $W_0$ acts on $x^\lambda$ \eqref{eq:xlam} by the action on the weight $\lambda$.
Also recall that for each $\mu \in (\frh_\bZ^*)_+ \subset \frh_\bZ^*$
we defined $w(\mu) \in t(\mu) W_0 \subset W$ by \eqref{eq:w(mu)}.

\begin{fct}[{\cite[Theorem 6.6]{S}}]\label{fct:Pmu}
For each dominant weight $\lambda \in (\frh_\bZ^*)_+$, the element
\[
 P_\lambda(x) := \frac{1}{t_{w_\mu}^{-\hf}W_{\lambda}(t)} U S_{w(\lambda)}^Y 1
           = \frac{1}{t_{w_\lambda}^{-\hf}W_{\mu}(t)} U E_{\lambda}(x) \in \wt{DH}(W)
\]
belongs to $\bK[x^{\pm1}]^{W_0}$. 
We call $P_\lambda(x)$ \emph{the} (\emph{monic}) 
\emph{Koornwinder polynomial associated to $\lambda$}.
\end{fct}

Note that the coefficient of $x^\lambda$ in $P_\lambda(x)$ is $1$ since the coefficient of 
the top term $x^\lambda$ in $U S_{w(\lambda)}^Y 1$ is $t_{w_\lambda}^{-\hf}W_{\lambda}(t)$.
To emphasize the root system $(C_n^\vee,C_n)$, we call $P_\lambda(x)$ 
\emph{the Koornwinder polynomial of rank $n$} or \emph{of type $(C_n^\vee,C_n)$}.

\section{Littlewood-Richardson coefficients}\label{s:LR}

Yip \cite[Theorem 4.4]{Y} derived a combinatorial explicit formula of LR coefficients 
for Macdonald polynomials $P_\lambda(x)$ in the case of untwisted affine root systems.
In this section, we derive a $(C_n^\vee,C_n)$-analogue of Yip's formula.
The outline of the derivation is quite similar to Yip's proof \cite[\S\S 3.1--4.1]{Y}, 
but we need non-trivial adjustments in each step.

\subsection{Products of non-symmetric Koornwinder polynomials and monomials}

In \cite[Theorem 3.3]{Y}, Yip derived an expansion formula for the product of 
the monomial $x^\mu$ and the non-symmetric Macdonald polynomial $E_\lambda(x)$ 
in the case of untwisted affine root systems.
In this subsection, we give its $(C_n^\vee,C_n)$-type analogue (Corollary \ref{cor:xmuE}).

We will use the notations in \S \ref{sss:AH:DH}.
In particular, $\wt{DH}(W)$ is the extension \eqref{eq:wDH} of 
the double affine Hecke algebra $DH(W)$ of type $(C_n^\vee,C_n)$,
$S^Y_i \in \wt{DH}(W)$ is the $Y$-intertwiner \eqref{eq:SYi}, 
and $S^Y_w$ for $w \in W$ is the product of $S^Y_i$'s \eqref{eq:SYw}.
We also denote the Bruhat order in $W$ by $\bleq$.

As a preparation of Proposition \ref{prp:xzSY},
we derive a product formula of the $Y$-intertwiners. 

\begin{prp}\label{prp:SYiw}
For $w \in W$ and $i=0,\ldots,n$, we have the following relations in $\wt{DH}(W)$.
\begin{enumerate}[nosep,label=(\roman*)]
\item 
If $w \bleq s_i w$, then $S^Y_iS^Y_w = S^Y_{s_i w}$.
\item
If $w \bgeq s_i w$, then $S^Y_iS^Y_w = n_i(Y^{-\alpha_i}) S^Y_{s_i w}$,
where 
\begin{align*}
&n_0(Y^{\beta}) :=
 \frac{(1-u_n^{ \hf} u_0^{ \hf} Y^{\frac{\beta}{2}})
       (1+u_n^{ \hf} u_0^{-\hf} Y^{\frac{\beta}{2}})}{1-Y^{\beta}}
 \frac{(1+u_n^{-\hf} u_0^{ \hf} Y^{\frac{\beta}{2}})
       (1-u_n^{-\hf} u_0^{-\hf} Y^{\frac{\beta}{2}})}{1-Y^{\beta}}
 \quad (\beta\in W.\alpha_0), \\
&n_i(Y^{\beta}) :=
 \frac{1-tY^{\beta}}{1-Y^{\beta}} \frac{1-t^{-1}Y^{\beta}}{1-Y^{\beta}}
 \hspace{16.5em} (\beta\in W.\alpha_i,\ 0<i<n), \\
&n_n(Y^{\beta}) :=
 \frac{(1-t_n^{ \hf} t_0^{ \hf} Y^{\frac{\beta}{2}})
       (1+t_n^{ \hf} t_0^{-\hf} Y^{\frac{\beta}{2}})}{1-Y^{\beta}}
 \frac{(1+t_n^{-\hf} t_0^{ \hf} Y^{\frac{\beta}{2}})
       (1-t_n^{-\hf} t_0^{-\hf} Y^{\frac{\beta}{2}})}{1-Y^{\beta}}
 \quad (\beta\in W.\alpha_n).
\end{align*}
\end{enumerate}
\end{prp}

\begin{proof}
Fix $w\in W$ and choose a reduced expression $w=s_{i_1}\cdots s_{i_r}$.
By the definitions \eqref{eq:SYw}, \eqref{eq:SYi} 
and the equation \eqref{eq:Sics}, we have
\begin{align*}
  S_w^Y
&=S_{i_1}^Y\cdots S_{i_r}^Y
 =(T_{i_1}^\vee+\psi^+_{i_1}(Y^{-\alpha_{i_1}})) \cdots 
  (T^\vee_{i_r}+\psi^+_{i_r}(Y^{-\alpha_{i_r}})) \\
&=c_{i_1}^\vee(Y^{-\alpha_{i_1}})s_{i_1} \cdots c_{i_r}^\vee(Y^{-\alpha_{i_r}})s_{i_r}
 =c_{i_1}^\vee(Y^{-\beta_1})\cdots c_{i_r}^\vee(Y^{-\beta_r})w.
\end{align*}
Here we set $\beta_k:=s_{i_1} \cdots s_{i_{k-1}}(\alpha_{i_k})$ ($k=1,\ldots,r$) and 
\[
 c^\vee_i(z):=\phi(c_i(z))=
 \begin{dcases}
 u_n^{-\hf}\frac{(1-u_n^{\hf}u_0^{\hf}z^\hf)(1+u_n^\hf u_0^{-\hf}z^\hf)}{1-z} &(i=0),\\
 t^{-\hf}\frac{1-tz}{1-z} & (0 < i < n), \\
 t_n^{-\hf}\frac{(1-t_n^{\hf}t_0^{\hf}z^\hf)(1+t_n^\hf t_0^{-\hf}z^\hf)}{1-z} &(i=n).
 \end{dcases}
\]
Since $w=s_{i_1} \cdots s_{i_r}$ is a reduced expression, 
we have $\beta_k\in S_+$ for $k=1,\dots,r$,
where $S_+ \subset S$ denotes the set of positive affine roots \eqref{eq:S+}.
The product $S_i^Y$ ($i=0,\ldots,n$) and $S_w^Y$ is now calculated as
\begin{align}\label{eq:SYiSYw}
 S_i^Y S_w^Y = 
 c^\vee_i(Y^{-\alpha_i}) s_i c_{i_1}^\vee(Y^{-\beta_1}) \cdots 
 c^\vee_{i_r}(Y^{-\beta_r}) w.
\end{align}
If $\ell(s_iw)=\ell(w)+1$, 
then the equation \eqref{eq:SYiSYw} becomes $S_i^Y S_w^Y=S_{s_i w}^Y$. 
If $\ell(s_iw)=\ell(w)-1$, then there exists $k\in\pr{1,\dots,r}$ such that 
$s_i(\beta_{k-1})\in S_+$ and $s_i(\beta_k)\in S_-$.
Since we have $\beta_k=\alpha_i$, the equation \eqref{eq:SYiSYw} becomes
\begin{align*}
 S_i^Y S_w^Y
 &=c_i^\vee(Y^{-\alpha_i})s_i c_{i_1}^\vee(Y^{-\beta_1}) \cdots 
   c_{i_r}^\vee(Y^{-\beta_r})w\\ 
 &=c_i^\vee(Y^{-\alpha_i}) c_{i_1}^\vee(Y^{-s_i(\beta_1)}) \cdots 
   c_{i_{k-1}}^\vee(Y^{-s_i(\beta_{k-1})})
   s_ic_{i_k}^\vee(Y^{-\alpha_i}) \cdots c_{i_r}^\vee(Y^{-\beta_r})w \\
 &=c_i^\vee(Y^{-\alpha_i}) c_i^\vee(Y^{\alpha_i}) c_{i_1}(Y^{-s_i(\beta_1)}) \cdots 
   \wh{c_{i_k}^\vee(Y^{\alpha_i})} \cdots c_{i_r}^\vee(Y^{-s_i(\beta_r)})s_i w 
  =c_i^\vee(Y^{-\alpha_i})c_i^\vee(Y^{\alpha_i})S_{s_iw}^Y.
\end{align*}
Here the symbol $\ \wh{\,} \ $ denotes skipping the term. 
Then the consequence follows from the equality 
$c_i^\vee(Y^{-\alpha_i})c_i^\vee(Y^{\alpha_i})=n_i(Y^{-\alpha_i})$,
which can be checked by a direct calculation.
\end{proof}

The same discussion shows the following statement. 

\begin{cor}\label{cor:SYwi}
For $w \in W$ and $i=0,\ldots,n$.
we have the following relations in $\wt{DH}(W)$
\begin{enumerate}[nosep,label=(\roman*)]
\item 
If $w \bleq w s_i$, then $S^Y_w S^Y_i = S^Y_{w s_i}$.
\item
If $w \bgeq w s_i$, then $S^Y_w S^Y_i = S^Y_{w s_i}n_i(Y^{-\alpha_i})$,
where $n_i(Y^{-\alpha_i})$ is given in Proposition \ref{prp:SYiw}. 
\end{enumerate}
\end{cor}

Next we recall the notations on alcove walks in \S \ref{sss:r:aw}.
Given $z,w \in W$ together with a reduced expression $z=s_{i_r}\cdots s_{i_1}$,
we defined the set $\Gamma(\oa{z},w)$ of alcove walks 
of type $\oa{z}=(i_r,\ldots,i_1)$ beginning at $w A$.
For an alcove walk $p=(p_0,\ldots,p_r) \in \Gamma(\oa{z},w)$,
the $k$-th step means the the transition from $p_{k-1}$ to $p_k$, 
which is classified into the four types in Table \ref{tab:sp}.

Now we define $x^z \in DH(W)$ for $z \in W$ with a chosen reduced expression 
$z=s_{i_r}\cdots s_{i_1}$. 
Let $q$ be the alcove walk given by
\[
 q := 
 (z A, \, z s_{i_1}A, \, z s_{i_1}s_{i_2}A, \, \ldots, \, z s_{i_1} \cdots s_{i_r}A=A) 
 \in \Gamma(\oa{z}^{-1},z).
\]
Here $\oa{z}^{-1}:=\oa{z^{-1}}=(i_1,\ldots,i_r)$.
Then we define $x^z$ by
\begin{align}\label{eq:xzT}
 x^z := (T_{i_r}^\vee)^{\ep_r}\cdots (T_{i_1}^\vee)^{\ep_1},
\end{align}
where $T_i^\vee := \phi(T_i) \in DH(W)$ as in \eqref{eq:Tiv}, 
and we set $\ep_k := 1$ if the $k$-th step is a positive crossing, 
and $\ep_k := -1$ if the $k$-th step is a negative crossing 
according to the classification in Table \ref{tab:sp}.

\begin{prp}\label{prp:xzSY}
Given $z,w\in W$ with a chosen reduced expression $z=s_{i_r}\cdots s_{i_1}$, we have
\[
 x^z S_w^Y = \sum_{p \in \Gamma(\oa{z}^{-1},w^{-1})} S_{e(p)^{-1}}^Y g_p(Y)n_p(Y)
\]
in $\wt{DH}(W)$, where $e(p) \in W$ is the element \eqref{eq:e(p)},
and the terms $g_p(Y)$ and $n_p(Y)$ are given by
\begin{align*}
&g_p(Y) := \prod_{k \in \varphi_-(p)}\bigl(-\psi_{i_k}^-(Y^{-h_k(p)})\bigr)
           \prod_{k \in \varphi_+(p)}
           \bigl(-\psi_{i_k}^+(Y^{-h_k(p)})\bigr), \\
&n_p(Y) := \prod_{k \in \xi_{\des}(p)}n_{i_k}(Y^{-h_k(p)}).
\end{align*}
Here $h_k(p)$ is given by \eqref{eq:hk(p)},
$\varphi_+(p)$ and $\varphi_-(p)$ are by \eqref{eq:vp(p)},
$\xi_{\des}(p)$ is by \eqref{eq:xp(p)},  
$\psi_i^{\pm}(z)=\phi(\varphi_i^{\pm}(z))$ is by \eqref{eq:psi(z)},
and $n_i(z)$ is given in Proposition \ref{prp:SYiw}.
\end{prp}

\begin{proof}
We show the statement by induction on the length of $z \in W$.
If $\ell(z)=0$, that is $z=e$,
then the right hand side consists only of the term for $p=(p_0=w A)$, 
so that it is equal to  $S^Y_w$, and we have the relation.

Next we assume $z \neq e$ and that 
the result holds for any element $w\in W$ such that $w<\ell(z)$.

Fix a reduced expression of $z$, 
and write it as $z=s_i \zeta$, $\zeta=s_{i_r} \cdots s_{i_1}$. 
By the hypothesis, we can write 
\begin{align}\label{eq:xzSY:0}
 x^z S_w^Y = (T_i^\vee)^{\ep} x^{\zeta} S_w^Y 
 = \sum_{p \in \Gamma(\oa{\zeta}^{-1},w^{-1})} 
(T_i^\vee)^{\ep} S_{e(p)^{-1}}^Y g_p(Y)n_p(Y).
\end{align}
Here $\ep \in \{\pm1\}$ is the sign determined by $z$.
Let us calculate the rightmost side. 
Take an element 
\[
 p = (w^{-1}A,\, w^{-1} s_{i_1}^{\ep_1} A, \, \ldots, \, 
      w^{-1} s_{i_1}^{\ep_1} \cdots s_{i_r}^{\ep_r} A) 
 \in \Gamma(\oa{\zeta}^{-1},w^{-1}).
\]
Since we have $(T_i^\vee)^{\pm1}=S_i^Y-\psi^{\pm}_i(Y^{-\alpha_i})$
by the definition \eqref{eq:Tiv} of $T_i^\vee$, 
the term contributed by $p$ becomes
\begin{align*}
(T_i^\vee)^{\ep}S_{e(p)^{-1}}^Y g_p(Y)n_p(Y)
&=(S_i^Y-\psi_i^{\ep}(Y^{-\alpha_i}))S_{e(p)^{-1}}^Y g_p(Y)n_p(Y) \\
&=S_i^Y S_{e(p)^{-1}}^Yg_p(Y)n_p(Y)
 +(-\psi_i^{\ep}(Y^{-\alpha_i}))S_{e(p)^{-1}}^Y g_p(Y)n_p(Y) \\
&=S_i^Y S_{e(p)^{-1}}^Y g_p(Y)n_p(Y)
 +S_{e(p)^{-1}}^Y(-\psi_i^{\ep}(Y^{-e(p)\alpha_i}))g_p(Y)n_p(Y).
\end{align*}
In the last equality we used \eqref{eq:SYl}.
We treat the two terms in the last line separately.

For the first term $S_i^Y S_{e(p)^{-1}}^Y g_p(Y)n_p(Y)$,
we further divide the argument into two cases according to the Bruhat order.
\begin{enumerate}[nosep,label=(\roman*)]
\item 
The case $e(p)^{-1} \bleq s_i e(p)^{-1}$.
By Proposition \ref{prp:SYiw}, we have
$S^Y_i S^Y_{e(p)^{-1}} = S^Y_{s_i e(p)^{-1}} = S_{e(p_1)^{-1}}$, 
where the alcove walk
\begin{align}\label{eq:xzSY:p_1}
 p_1  = (w^{-1}A,\, w^{-1} s_{i_1}^{\ep_1} A, \, \ldots, \, 
         w^{-1} s_{i_1}^{\ep_1} \cdots s_{i_r}^{\ep_r} A, \, 
         w^{-1} s_{i_1}^{\ep_1} \cdots s_{i_r}^{\ep_r} s_i^{\ep} A) 
 \in \Gamma(\oa{z},w^{-1})
\end{align}
is an extension of $p$ by a crossing (Table \ref{tab:sp}).
By the hypothesis $e(p)^{-1} \bleq s_i e(p)^{-1}$, the last step of $p_1$ is an ascent,
and we have 
$\varphi_+(p_1)=\varphi_+(p)$, $\varphi_-(p_1)=\varphi_-(p)$ and 
$\xi_{\des}(p_1)=\xi_{\des}(p)$. 
Thus we have $g_p(Y)n_p(Y)=g_{p_1}(Y)n_{p_1}(Y)$ 
and $S_i^Y S_{e(p)^{-1}}^Y g_p(Y)n_p(Y)=S_{e(p_1)^{-1}}g_{p_1}(Y)n_{p_1}(Y)$.

\item 
The case $e(p)^{-1} \bgeq s_i e(p)^{-1}$. 
By Propositions \ref{prp:SYiw}, we have
\[
 S_i^Y S_{e(p)^{-1}}^Y = n_i(Y^{-\alpha_i})S_{s_i e(p)^{-1}}^Y = 
 n_i(Y^{-\alpha_i})S_{e(p_1)^{-1}}^Y = S_{e(p_1)^{-1}}^Y n_i(Y^{-e(p_1)\alpha_i}).
\]
Here $p_1 \in \Gamma(\oa{z},w^{-1})$ is the same as \eqref{eq:xzSY:p_1},
but in this case the last step is a descent crossing,
and the hyperplane crossed by the last step is $H_{e(p_1)\alpha_i}$ since
\[
 h_{r+1}(p_1)=-e(p)(\alpha_i)=-e(p_1)s_i(\alpha_i)=e(p_1)\alpha_i.
\]
We then have $\xi_{\des}(p_1)=\xi_{\des}(p)\cup\{r+1\}$ and
$n_{p_1}(Y)=n_{p}(Y) n_{i}(Y^{-h_{r+1}(p_1)})$.
Combining them with $\varphi_+(p_1)=\varphi_+(p)$ and $\varphi_-(p_1)=\varphi_-(p)$,
we have $n_i(Y^{-e(p_1)\alpha_i}) g_p(Y)n_p(Y)=g_{p_1}(Y)n_{p_1}(Y)$.
Hence also in this case, we have
$S_i^Y S_{e(p)^{-1}}^Y g_p(Y)n_p(Y)=S_{e(p_1)^{-1}}g_{p_1}(Y)n_{p_1}(Y)$.
\end{enumerate} 
Taking the summation over $p$, we therefore have
\begin{align}\label{eq:xzSY:1}
 \sum_{p \in \Gamma(\oa{\zeta}^{-1},w^{-1})} S_i^Y S_{e(p)^{-1}}^Y g_p(Y)n_p(Y) = 
 \sum_{\substack{p_1 \in \Gamma(\oa{z}^{-1},w^{-1}), 
       \\ \text{the last step is a crossing}}} S_{e(p_1)^{-1}}^Y g_{p_1}(Y)n_{p_1}(Y).
\end{align}

Next we consider the term $S_{e(p)^{-1}}^Y(-\psi_i^{\ep}(Y^{-e(p)\alpha_i}))g_p(Y)n_p(Y)$.
We make a similar argument as in the first term,
and here we use the alcove walk $p_2 \in \Gamma(\oa{z},w^{-1})$ 
which is an extension of $p$ by a folding.
We have 
$e(p_2)=e(p)$, $\varphi_+(p_2)=\varphi_+(p) \cup \{r+1\}$ and
$\xi_{\des}(p_2)=\xi_{\des}(p)$.
Using $p_2$ we have
$S_{e(p)^{-1}}^Y(-\psi_i^{\ep}(Y^{-e(p)\alpha_i}))g_p(Y)n_p(Y) = 
 S_{e(p_2)^{-1}}^Y g_{p_2}(Y)n_{p_2}(Y)$.
We therefore have
\begin{align}\label{eq:xzSY:2}
 \sum_{p \in \Gamma(\oa{\zeta}^{-1},w^{-1})} 
 S_{e(p)^{-1}}^Y(-\psi_i^{\ep}(Y^{-e(p)\alpha_i}))g_p(Y)n_p(Y) = 
 \sum_{\substack{p_2 \in \Gamma(\oa{z}^{-1},w^{-1}), 
       \\ \text{the last step is a folding}}} 
 S_{e(p_2)^{-1}}^Y g_{p_2}(Y)n_{p_2}(Y).
\end{align}

By \eqref{eq:xzSY:0} and \eqref{eq:xzSY:1}, \eqref{eq:xzSY:2}, we have
$x^z S_w^Y=\sum_{p \in \Gamma(\oa{z}^{-1},w^{-1})} S_{e(p)^{-1}}^Y g_p(Y)n_p(Y)$.
Hence the induction step is proved.
\end{proof}

The definition \eqref{eq:xzT} of $x^z$ for $z \in W$ and 
the definition \eqref{eq:xlam} of $x^\mu$ for $\mu \in P = \frh_\bZ^*$
are consistent in the following sense.
Recall that we denote by $t(\mu) \in t(P) \subset W$ 
the element associated to $\mu \in P = \frh_\bZ^*$.

\begin{lem}\label{lem:xmu=xtmu}
We have $x^{t(\mu)}=x^\mu$ for $\mu \in P = \frh_\bZ^*$.
In particular, we have $x^{t(\ep_i)}=x_i$ for $i=1,\ldots,n$.
\end{lem}

\begin{proof}
It is enough to show the latter half.
By \eqref{eq:Yred}, we have
\[
 Y_i^{-1}= T_{i}^{-1} \cdots T_{n-1}^{-1} T_n^{-1} T_{n-1}^{-1}
                      \cdots T_1^{-1} T_0^{-1} T_1 \cdots T_{i-1}
 \quad(i=1,\dots,n).
\]
Applying the anti-involution $\phi$ \eqref{eq:inv} to these. 
\[
 x_i = \phi(Y_i^{-1})
     = T_{i-1}^{\vee} \cdots T_1^\vee (T_0^\vee)^{-1} (T_1^\vee)^{-1} \cdots 
      (T_{n-1}^\vee)^{-1} (T_n^\vee)^{-1} (T_{n-1}^\vee)^{-1} \cdots (T_i^\vee)^{-1}.
\]
On the other hand, we can calculate $x^{t(\ep_i)}$ 
directly by Definition \eqref{eq:xzT},
and can check $x_i=x^{t(\ep_i)}$.
\end{proof} 

We denote the \emph{dominant chamber} for the weight lattice by
\[
 C := \pr{x\in\frh_\bR^* \mid \br{\alpha^\vee,x}>0,\ \alpha\in R_+}.
\]
As for the fundamental alcove $A$ \eqref{eq:FA}, we have $A\subset C$.

Let $v, w \in W$, and choose a reduced expression $v=s_{i_1} \cdots s_{i_r}$ of $v$.
If an alcove walk $p \in \Gamma(\oa{v},w)$
satisfies $e(p)^{-1}A\subset C$, where $e(p) \in W$ is the element \eqref{eq:e(p)},
then using the $W$-valued function $w(\,)$ in \eqref{eq:w(mu)}, 
we define $\varpi(p) \in (\frh_\bZ^*)_+$ by the relation 
\begin{align}\label{eq:varpi}
 e(p)^{-1}=w(\varpi(p)).
\end{align} 
Also we define $\Gamma^{C}(\oa{v},w) \subset \Gamma(\oa{v},w)$ by 
\begin{align}\label{eq:GammaC}
 \Gamma^{C}(\oa{v},w):=
 \pr{p=(p_0,\ldots,p_r)\in\Gamma(\oa{v},w) \mid p_i \in C,\ \forall \, i=0,\ldots,r }.
\end{align}

Using these symbols, we have the following corollary 
of Proposition \ref{prp:xzSY}.

\begin{cor}[{c.f.\ \cite[Corollary 4.1]{Y}}]\label{cor:xmuE}
Let $\lambda,\mu\in\frh_\bZ^*$, and fix a reduced expression $t(\lambda)=s_{i_r}\cdots s_{i_1}$.
Then we have
\begin{align*}
&x^{\lambda}E_\mu(x) = 
 \sum_{p \in\Gamma^{C} (\oa{t(-\lambda)},w(\mu)^{-1})} g_p n_p E_{\varpi(p)}(x), \\
&g_p := \prod_{k \in \varphi_-(p)}
        \bigl(-\psi_{i_k}^-(q^{\sh(-h_k(p))}t^{\hgt(-h_k(p))})\bigr)
        \prod_{k \in \varphi_+(p)}
        \bigl(-\psi_{i_k}^+(q^{\sh(-h_k(p))}t^{\hgt(-h_k(p))})\bigr), \\
&n_p := \prod_{k \in \xi_{\des}(p)}n_{i_k}(q^{\sh(-h_k(p))}t^{\hgt(-h_k(p))}).
\end{align*}
\end{cor}

\begin{proof}
We apply 
$x^z S_w^Y=\sum_{p \in \Gamma(\oa{z}^{-1},w^{-1})} S_{e(p)^{-1}}^Y g_p(Y)n_p(Y)$ 
in Proposition \ref{prp:xzSY} to $z=t(\lambda)$ and $w=w(\mu)$.
Since $x^{t(\lambda)}=x^{\lambda}$ by Lemma \ref{lem:xmu=xtmu}, we have
\[
 x^\lambda S_{w(\mu)}^Y =
 \sum_{p \in \Gamma(\oa{t(-\lambda)},w(\mu)^{-1})} S_{e(p)^{-1}}^Y g_p(Y)n_p(Y).
\]
Taking the product of each side with $1$ and  
using the definition of the non-symmetric Koornwinder polynomial $E_\mu(x)$
(Fact \ref{fct:Emu}) and the equality 
$Y^\beta 1 = q^{\sh(\beta)}t^{\hgt(\beta)}$ in \eqref{eq:Ylam1}, we have
\[
 x^\lambda E_{\mu}(x) =  
 \sum_{p \in \Gamma(\oa{t(-\lambda)},w(\mu)^{-1})} g_p n_p S_{e(p)^{-1}}^Y 1.
\]
Next we consider the condition under which
the factor $n_{i_k}(q^{\sh(-h_k(p))}t^{\hgt(-h_k(p))})$ in $n_p$ vanishes.
By the definition of the factor (Proposition \ref{prp:SYiw}), the condition is
$q^{\sh(-h_k(p))}t^{\hgt(-h_k(p))}=t^{\pm1}$ ($i_k=1,\ldots,n-1$)
and $q^{\sh(-h_k(p))}t^{\hgt(-h_k(p))}=t_0^{\pm1}t_n^{\pm1}$ ($i_k=n$).
Then by the definition \eqref{eq:hk(p)} of $h_k(p)$, the alcove walk $p$ 
that contributes to the summation is contained in the dominant chamber $C$.
Now the consequence follows from the definition of $E_\mu(x)$ and 
and that \eqref{eq:varpi} of $\varpi(p)$.
\end{proof}

\subsection{Some lemmas}

In this subsection we prepare some lemmas for the symmetrizer $U$ and 
the Koornwinder polynomials $P_\lambda(x)$, 
which are $(C^\vee_n,C_n)$-type analogue of \cite[Proposition 3.6]{Y}.  

\begin{lem}[{c.f.\ \cite[Proposition 3.6 (a)]{Y}}]\label{lem:U}
The symmetrizer $U$ \eqref{eq:U} has the following expression.
\begin{align}
&\nonumber
 U=\sum_{w\in W_0} S_w^Y 
   \prod_{\alpha \in \cL(w^{-1},w_0^{-1})}b(Y^{-\alpha}), \\
&\label{eq:U:b}
 b(Y^{-\alpha}) := 
 \begin{dcases}
 t^\hf \frac{1-t^{-1}Y^{-\alpha}}{1-Y^{-\alpha}} 
 & (\alpha \not\in W_0.\alpha_n) \\
 t_n^{\hf} \frac{(1+t_0^{ \hf}t_n^{-\hf}Y^{-\frac{\alpha}{2}})
                 (1-t_0^{-\hf}t_n^{-\hf}Y^{-\frac{\alpha}{2}})}
                {1-Y^{-\alpha}}
 & (\alpha \in W_0.\alpha_n)
 \end{dcases}.
\end{align}
Here $\cL(v,w) \subset S$ is given by \eqref{eq:cL(v,w)}, and 
$w_0 \in W_0$ is the longest element \eqref{eq:wmu}.
\end{lem}

\begin{proof}
By the definition of $U$ and the definition \eqref{eq:SYw} of the $Y$-intertwiner $S_w^Y$, 
we can expand $U$ as 
\[
 U = \sum_{w\in W_0}S_w^Y b_w(Y), \quad b_w(Y) \in \bK(Y).
\]
For the longest element $w_0 \in W_0$, 
the coefficient of $T_{w_0}$ in $U$ is $1$, and thus we have $b_{w_0}(Y)=1$. 
We calculate the term $b_w(Y)$ for $w \in W_0 \setminus \{w_0\}$ 
by induction on the length $\ell(w)$.
Assume $b_v(Y)=\prod_{\alpha \in \cL(v^{-1},w_0^{-1})}b(Y^{-\alpha})$
for any element $v \in W_0$ satisfying $\ell(v)>\ell(w)$.
By the equality $U T_i=U t_i^\hf$ ($i=1,\dots,n$) in \eqref{eq:UTUt} and 
the definition \eqref{eq:SYi} of $S^Y_i$, 
we have 
\begin{align}\label{eq:U:rec}
  \sum_{w\in W_0}S_w^Yb_w(Y)t_i^\hf = U t_i^\hf = U T_i
= \sum_{w\in W_0}S_w^Yb_w(Y)T_i 
= \sum_{w\in W_0}S_w^Yb_w(Y)(S_i^Y-\psi^+_i(Y^{-\alpha_i})).
\end{align}
Now note that for $w \neq w_0$ there exists an index $i=1,\ldots,n$ 
such that $w \bleq v:=w s_i$. 
Taking this index $i$ and comparing the coefficients of $S^Y_w$ in the equality 
\eqref{eq:U:rec} with the help of \eqref{eq:SYl} and Corollary \ref{cor:SYwi},
we have $b_v(Y) t_i^{\hf} = b_w(s_i.Y) - b_v(Y) \psi_i^+(Y^{-\alpha_i})$.
Here $b_w(s_i.Y)$ is obtained from $b_w(Y)$ by replacing 
$Y^\lambda$ with $Y^{s_i.\lambda}$.
Then by the definition \eqref{eq:psi(z)} of $\psi^+_i(z)$ we have
\begin{align*}
  b_w(Y) /b_v(s_i.Y)
&=t_i^{\hf}+\psi_i^+(Y^{-s_i.\alpha_i}) = t_i^{\hf}+\psi_i^+(Y^{\alpha_i}) \\
&=\begin{cases}
   t_i^\hf \frac{1-t^{-1}Y^{-\alpha_i}}{1-Y^{-\alpha_i}} & (0<i<n) \\
   t_n^\hf 
   \frac{(1+t_0^{ \hf}t_n^{-\hf}Y^{-\frac{\alpha_n}{2}})
         (1-t_0^{-\hf}t_n^{-\hf}Y^{-\frac{\alpha_n}{2}})}{1-Y^{-\alpha_n}} & (i=n)
  \end{cases},
\end{align*}
so that it is equal to $b(Y^{-\alpha_i})$.
On the other hand, by \eqref{eq:cL:trs} we have 
$\cL(w^{-1},w_0^{-1})=s_i.\cL(v^{-1},w_0^{-1}) \sqcup \{\alpha_i\}$.
Thus we have $b_w(Y) = b_v(s_i.Y) b(Y^{-\alpha_i})
=\prod_{\alpha \in \cL(w^{-1},w_0^{-1})}b(Y^{-\alpha})$.

\end{proof}

We can apply the argument of the proof to the stabilizer $W_\mu \subset W_0$ 
for a dominant weight $\mu \in (\frh_\bZ^*)_+$ instead of $W_0$. 
As a result, we have the following claim. 

\begin{cor}\label{cor:Umu}
For each $\mu \in (\frh_\bZ^*)_+$, we have 
\[
 \sum_{u \in W_\mu}t_{w_\mu u}^{-\hf}T_u = 
 \sum_{w \in W_\mu} S_w^Y \prod_{\alpha \in \cL(w^{-1},w_\mu^{-1})}b(Y^{-\alpha}).
\]
Here $b(Y^{-\alpha}) \in \bK(Y)$ is given by \eqref{eq:U:b}.
\end{cor}

For a dominant weight $\mu \in (\frh_\bZ^*)_+$, we denote by
\begin{align}\label{eq:W^mu}
 W^\mu \subset W_0
\end{align}
the complete system of representatives of the quotient set $W_0/W_\mu$ 
consisting of the shortest elements.
We also denote by $v_\mu \in W^\mu$ its longest element.

Now let us recall the element $w(\mu) \in t(\mu) W_0 \subset W$ in the \eqref{eq:w(mu)}.
We then have the following lemma for the Koornwinder polynomial $P_\mu(x)$ (Fact \ref{fct:Pmu})
and the non-symmetric Koornwinder polynomial $E_\mu(x)$ (Fact \ref{fct:Emu}).

\begin{lem}[{c.f.\ \cite[Proposition 3.6 (b)]{Y}}]\label{lem:PmuEvmu}
For $\lambda \in (\frh_\bZ^*)_+$ we have
\begin{align*}
&P_\lambda(x) = \sum_{v \in W^\lambda}
 \Bigl[\prod_{\alpha \in w(\lambda)^{-1}\cL(v^{-1},v_\lambda^{-1})} \rho(\alpha)\Bigr]
 E_{v.\lambda}(x), \\
&\rho(\alpha) := 
 \begin{dcases}
  t^\hf 
  \frac{1-t^{-1}q^{\sh(-\alpha)}t^{\hgt(-\alpha)}}
       {1-q^{\sh(-\alpha)}t^{\hgt(-\alpha)}} & (\alpha \not\in W.\alpha_n) \\
  t_n^\hf
  \frac{(1+t_0^{\hf}t_n^{-\hf}
        q^{\hf\sh(-\alpha)}t^{\hf\hgt(-\alpha)})
        (1-t_0^{-\hf}t_n^{-\hf}
        q^{\hf\sh(-\alpha)} t^{\hf\hgt(-\alpha)})}
       {1-q^{\sh(-\alpha)}t^{\hgt(-\alpha)}} & (\alpha \in W.\alpha_n)
 \end{dcases},
\end{align*}
where $\sh(\beta)$ and $\hgt(\beta)$ for $\beta \in S$ are given by \eqref{eq:shhgt}.
\end{lem}

\begin{proof}
We write Lemma \ref{lem:U} as 
\begin{align*}
 U=\sum_{w\in W_0}S_w^Y b_{(w^{-1},w_0^{-1})}(Y), \quad
 b_{(w^{-1},w_0^{-1})}(Y):=\prod_{\alpha \in \cL(w^{-1},w_0^{-1})}b(Y^{-\alpha}).
\end{align*}
Since $W^\lambda$ consists of representatives of $W_0/W_\lambda$,
there exist $v \in W^\lambda$ and $u \in W_\lambda$ uniquely such that $w=v u$.
Using Corollary \ref{cor:Umu}, we have 
\begin{align*}
U&=\sum_{w \in W_0}S_w^Y b_{(w^{-1},w_0^{-1})}(Y)
  =\Bigl[\sum_{v \in W^\lambda}S_v^Y b_{(v^{-1},v_\lambda^{-1})}(Y)\Bigr]
   \Bigl[\sum_{u \in W_\lambda}S_u^Y b_{(u^{-1},w_\lambda^{-1})}(Y)\Bigr] \\
 &=\Bigl[\sum_{v \in W^\lambda}S_v^Y b_{(v^{-1},v_\lambda^{-1})}(Y)\Bigr]
   \Bigl[\sum_{u \in W_\lambda}t_{w_\lambda u}^{-\hf}T_u\Bigr].
\end{align*}
The product with $S_{w(\lambda)}^Y 1$ gives 
\begin{align}\label{eq:PE:US1}
  U S_{w(\lambda)}^Y 1
&=\Bigl[\sum_{v\in W^\lambda}S_v^Y b_{(v^{-1},v_\lambda^{-1})}(Y)\Bigr]
  \Bigl[\sum_{u\in W_\lambda}t_{w_\lambda u}^{-\hf}T_u\Bigr] S_{w(\lambda)}^Y 1
 =t_{w_\lambda}^{-\hf}W_\lambda(t) 
  \sum_{v\in W^\lambda}S_v^Y b_{(v^{-1},v_\lambda^{-1})}(Y) S_{w(\lambda)}^Y 1,
\end{align}
where in the second equality we used the Poincar\'e polynomial \eqref{eq:Wmu(t)} 
and the relation $(T_u f)1=t_u^\hf f$ for $u\in W_0$ and $f \in \bK[x^{\pm1}]$ 
satisfying $u(f)=f$.
The latter relation is shown as follows.
If $s_i f=f$ for some $i=1,\dots,n$, then we have 
$(T_i-t_i^\hf)f=c_i(x^{\alpha_i})(s_i-1)f=0$, and so $T_i f=t_i^\hf f$. 
Now the relation follows by induction on the length of $u \in W_0$.

Let us continue the calculation \eqref{eq:PE:US1}.
Note that we have $v w(\lambda)=w(v.\lambda)$ for $v \in W^\lambda$.
By this relation and \eqref{eq:SYl}, 
each term in the right hand side of \eqref{eq:PE:US1} becomes
\begin{align*}
  S_v^Y b_{(v^{-1},v_\lambda^{-1})}(Y) S_{w(\lambda)}^Y 1
&=S_v^Y S_{w(\lambda)}^Y b_{(v^{-1},v_\lambda^{-1})}(w(\lambda)^{-1}.Y) 1
 =\bigl(S_{w(v.\lambda)}^Y 1\bigr)\bigl(b_{(v^{-1},v_\lambda^{-1})}(w(\lambda)^{-1}.Y) 1\bigr).
\end{align*}
Here $b_{(v^{-1},v_\lambda^{-1})}(w(\lambda)^{-1}.Y)$ is obtained from 
$b_{(v^{-1},v_\lambda^{-1})}(Y)$ by replacing $Y^\mu$ with $Y^{w(\lambda)^{-1}.\mu}$.
Now let us recall the equality 
$Y^\alpha 1 = q^{\sh(\alpha)} t^{\hgt(\alpha)}$ in \eqref{eq:Ylam1}.
Then we have $b(Y^{-\alpha})1=\rho(\alpha)$, and therefore 
\[
 b_{(v^{-1},v_\lambda^{-1})}(w(\lambda)^{-1}.Y) 1 
=\prod_{\alpha \in \cL(v^{-1},v_\lambda^{-1})} \bigl(b(Y^{-w(\lambda)^{-1}.\alpha})1\bigr)
=\prod_{\alpha \in w(\lambda)^{-1}\cL(v^{-1},v_\lambda^{-1})} \rho(\alpha).
\]

By summing over $v \in W^\lambda$ we have
\begin{align*}
 U S_{w(\lambda)}^Y 1
=t_{w_\lambda}^{-\hf} W_\lambda(t) \sum_{v\in W^\lambda} 
 \Bigl[\prod_{\alpha \in w(\lambda)^{-1}\cL(v^{-1},v_\lambda^{-1})} \rho(\alpha)\Bigr] 
 E_{v.\lambda}(x),
\end{align*}	
Now the result follows from the definition of $P_\lambda(x)$ (Fact \ref{fct:Pmu}).
\end{proof}

\subsection{Ram-Yip type formula and its application}

In \cite[Theorem 4.2]{Y}, Yip derived an expansion formula 
$E_\mu(x)P_\lambda(x)=\sum_\nu a_{\lambda,\mu}^\nu E_\nu(x)$ 
for the product of the non-symmetric Macdonald polynomial $E_\mu(x)$ and 
the Macdonald polynomial $P_\lambda(x)$ in the case of untwisted affine root systems.
In this subsection, we give its $(C^\vee_n,C_n)$-type analogue, i.e., 
an expansion formula for the product of the non-symmetric Koornwinder polynomial and 
the Koornwinder polynomial (Proposition \ref{prp:EPE}).
 
As a preparation, we cite the explicit formula of the non-symmetric Koornwinder polynomial
via alcove walks derived by Orr and Shimozono \cite{OS}.
It is a $(C^\vee_n,C_n)$-analogue of the explicit formula 
of the non-symmetric Macdonald polynomial in the untwisted affine root systems
derived by Ram and Yip \cite{RY}.
Let us call these formulas \emph{Ram-Yip type formulas}. 

We prepare the necessary notations for the explanation.
Let us given $v, w \in W$ and a reduced expression of $w$.
For an alcove walk $p \in \Gamma(\oa{w},z)$,
we denote the decomposition of the element $e(p) \in W$ \eqref{eq:e(p)}
with respect to the presentation $W=t(P)\rtimes W_0$ by
\begin{align}\label{eq:wtdr}
 e(p) = t(\wgt(p)) \dir(p), \quad \dir(p) \in W_0, \ \wgt(p) \in \frh_\bZ^*.
\end{align}

\begin{fct}[{\cite[Theorem 3.1]{RY}, \cite[Theorem 3.13]{OS}}]\label{fct:RY}
For $\mu \in \frh_\bZ^*$, 
let $w(\mu)$ be the shortest element among $t(\mu)W_0 \subset W$ \eqref{eq:w(mu)},
and fix its reduced expression $w(\mu)=s_{i_1}\cdots s_{i_r}$.
Then we have
\begin{align*}
&E_\mu(x)=\sum_{p\in\Gamma(\oa{w(\mu)},e)} f_p
 t_{\dir(p)}^{\hf}x^{\wgt(p)}, \\
&f_p  := 
 \prod_{k\in\varphi_+(p)}\psi_{i_k}^{+}(q^{\sh(-\beta_k)} t^{\hgt(-\beta_k))})
 \prod_{k\in\varphi_-(p)}\psi_{i_k}^{-}(q^{\sh(-\beta_k))} t^{\hgt(-\beta_k)}),
\end{align*}
where we set $\beta_k := s_{i_r}\cdots s_{i_{k+1}}(\alpha_{i_r})$ for $k=1,\dots,r$.
\end{fct}

Next we introduce some notations necessary for Proposition \ref{prp:EPE},
which are basically the ones in \cite[\S 4.1]{Y}.
Let us given $v,w \in W$ and a reduced expression $v=s_{i_1}\cdots s_{i_r}$.
Recall the set $\Gamma^C(\oa{v},w)$ of alcove walks belonging to the dominant chamber $C$
as in \eqref{eq:GammaC}.
Consider an alcove walk in $\Gamma^C(\oa{v},w)$ 
together with coloring of all the folding steps by either black or gray.
We call such a data a \emph{colored alcove walk}, and denote by
\begin{align}\label{eq:GC2}
 \Gamma^C_2(\oa{v},w)
\end{align}
the set of colored alcove walks arising from alcove walks in $\Gamma^C(\oa{v},w)$.

For a colored alcove walk $p \in \Gamma^C_2(\oa{v},w)$, we denote by
\begin{align}\label{eq:p*}
 p^* \in \Gamma(\oa{v}^{-1},w^{-1}e(p))
\end{align}
the uncolored alcove walk obtained by straightening all the gray foldings steps of $p$
and by translation so that it ends at $e(p^*)=e \in W$.
More explicitly, for a colored positive walk $p \in  \Gamma^C_2(\oa{v},w)$ with 
\[
 p = (w A, w s_{i_1}^{b_1}A, \ldots, w s_{i_1}^{b_1} \cdots s_{i_r}^{b_r}A),
\]
we define $\wt{p}_k$ for $k=1,\ldots,r$ as follows, 
according to whether the $k$-th step 
$p_{k-1} = w s_{i_1}^{b_1} \cdots s_{i_{k-1}}^{b_{k-1}} A \to 
     p_k = w s_{i_1}^{b_1} \cdots s_{i_k}^{b_k} A$ 
is a gray folding step or not:
\[
 \wt{p}_k := 
 \begin{cases}
 w s_{i_1}^{b_1} \cdots s_{i_{k-1}}^{b_{k-1}} s_{i_k} A 
 & (\text{$p_{k-1} \to p_k$ is a gray folding step}) \\
 p_k & (\text{otherwise})
 \end{cases}.
\]
Thus we obtain a new uncolored alcove walk 
$\wt{p}=(\wt{p}_0,\ldots,\wt{p}_r) \in \Gamma(\oa{v},w)$,
which was called the one obtained ``by straightening all the gray foldings".
Next we denote by $(c_1,\ldots,c_r) \in \{0,1\}^r$ the bit sequence 
corresponding to $\wt{p}$. In other words,
we have $\wt{p}=(w A, \ldots, w s_{i_1}^{c_1} \cdots s_{i_r}^{c_r}A)$.
Now the alcove walk $p^*$ is obtained by reversing 
the order of $\wt{p}$ and 
 translating the start to $w^{-1}e(\wt{p})$.
Explicitly, we have 
\[
  p^* := (s_{i_1}^{c_1} \cdots s_{i_r}^{c_r}A, 
          s_{i_1}^{c_1} \cdots s_{i_{r-1}}^{c_{r-1}}A, \ldots, s_{i_1}^{c_1} A, A).
\]

\begin{prp}[{c.f.\ \cite[Theorem 4.2]{Y}}]\label{prp:EPE}
For a weight $\mu \in \frh_\bZ^*$, 
we take a reduced expression $w(\mu)=s_{i_r}\cdots s_{i_1}$ of
$w(\mu) \in t(\mu) W_0 \subset W$.
Then for any dominant weight $\lambda \in (\frh_\bZ^*)_+$ we have
\[
 E_\mu(x) P_\lambda(x) = \sum_{v\in W^\lambda}  
 \sum_{p\in\Gamma^C_2(\oa{w(\mu)}^{-1},(v w(\lambda))^{-1})} 
 A_p C_p E_{\varpi(p)}(x).
\]
Here $W^\lambda$ is given by \eqref{eq:W^mu}, and the term $A_p$ is given 
with the help of $\rho(\alpha)$ in Lemma \ref{lem:PmuEvmu} by
\begin{align*}
A_p &:=
\prod_{\alpha \in w(\lambda)^{-1}\cL(v^{-1},v_\lambda^{-1})} \rho(\alpha), \\
\rho(\alpha) &:= 
 \begin{dcases}
  t^\hf 
  \frac{1-t^{-1}q^{\sh(-\alpha)}t^{\hgt(-\alpha)}}
       {1-q^{\sh(-\alpha)}t^{\hgt(-\alpha)}} & (\alpha \not\in W.\alpha_n) \\
  t_n^\hf
  \frac{(1+t_0^{\hf}t_n^{-\hf}q^{\hf\sh(-\alpha)}t^{\hf\hgt(-\alpha)})
        (1-t_0^{-\hf}t_n^{-\hf}
        q^{\hf\sh(-\alpha)} t^{\hf\hgt(-\alpha)})}
       {1-q^{\sh(-\alpha)}t^{\hgt(-\alpha)}} & (\alpha \in W.\alpha_n)
 \end{dcases}.
\end{align*}
The term $C_p$ is given by $C_p:=\prod_{k=1}^r C_{p,k}$, whose factor $C_{p,k}$ is 
determined by the $k$-th step of $p$ as follows.
\begin{align*}
C_{p,k} := \begin{cases}
1
&\text{the $k$-th step of $p$ is a positive crossing}\\
\prod_{k \in \xi_{\des}(p)}n_{i_k}(q^{\sh(-h_k(p))}t^{\hgt(-h_k(p))}) 
&\text{a negative crossing}\\
-\psi^+_{i_k}(q^{\sh(-h_k(p))}t^{\hgt(-h_k(p))})
&\text{a gray positive folding}\\
-\psi^-_{i_k}(q^{\sh(-h_k(p))}t^{\hgt(-h_k(p))})
&\text{a gray negative folding}\\
\psi^+_{i_k}(q^{\sh(-\beta_k)}t^{\hgt(-\beta_k)})
&\text{a black folding and the $k$-th step of $p^*$ is positive}\\
\psi^-_{i_k}(q^{\sh(-\beta_k)}t^{\hgt(-\beta_k)})
&\text{a black folding and the $k$-th step of $p^*$ is negative}
\end{cases},
\end{align*}
where $n_i(Y^{\beta})$ is given by Proposition \ref{prp:SYiw},
$\psi_{i_k}^\pm(z)$ is given by \eqref{eq:psi(z)} and
$h_k(p)$ is given by \eqref{eq:hk(p)}.
We also used $\beta_k := s_{i_1}\cdots s_{i_{r-1}}(\alpha_{i_r})$ for $k=1,\dots r$.
Finally $\varpi(p)$ is given by \eqref{eq:varpi}.
\end{prp}

Note that the term $A_p$ actually depends only on $v \in W^\mu$, 
which corresponds to the beginning of the colored alcove walk $p$.

\begin{proof}
On the Ram-Yip type formula 
$E_\mu(x) = \sum_{h \in \Gamma(\oa{w(\mu)},e)} f_h t_{\dir(h)}^\hf x^{\wgt(h)}$
(Fact \ref{fct:RY}),
let us act $U S_{w(\lambda)}^Y 1$ from the left.
Then we have
\begin{align*}
 E_\mu(x) U S_{w(\lambda)}^Y 1
=\bigl[\sum_{h\in\Gamma(\oa{w(\mu)},e)} f_h t_{\dir(h)}^\hf x^{\wgt(h)}\bigr]
 U S_{w(\lambda)}^Y 1
=\sum_{h \in \Gamma(\oa{w(\mu)},e)}f_h x^{e(h)} U S_{w(\lambda)}^Y 1.
\end{align*}
Here the second equality follows from 
the definition \eqref{eq:wtdr} of $\wgt(h)$ and $\dir(h)$,
as well as from the relation $T_i U = t_i^\hf U$ in \eqref{eq:UTUt}.
Moreover, by Lemma \ref{lem:PmuEvmu} and using the notation in its proof, we have
\begin{align}\label{eq:EUS1}
\begin{split}
E_\mu(x) U S_{w(\lambda)}^Y 1
&=\sum_{h \in \Gamma(\oa{w(\mu)},e)}f_h x^{e(h)} 
  \bigl[ t_{w_\lambda}^{-\hf}W_\lambda(t) \sum_{v \in W^\lambda}
         S_v^Y b_{(v^{-1},v_\lambda^{-1})}(Y) S_{w(\lambda)}^Y 1 \bigr]\\
&=t_{w_\lambda}^{-\hf}W_\lambda(t) \sum_{v \in W^\lambda} 
  \sum_{h \in \Gamma(\oa{w(\mu)},e)}
  f_h S_v^Y S_{w(\lambda)}^Y b_{(v^{-1},v_\lambda^{-1})}(w(\lambda)^{-1}.Y) 1
\\
&=t_{w_\lambda}^{-\hf}W_\lambda(t) \sum_{v \in W^\lambda} A_p
  \sum_{h \in \Gamma(\oa{w(\mu)},e)}f_h x^{e(h)} S_{v w(\lambda)}^Y 1.
\end{split}
\end{align}
Here we set
$A_p :=b_{((v w(\lambda))^{-1},t(-w_0\lambda))}= 
 \prod_{\alpha \in w(\lambda)^{-1}\cL(v^{-1},v_\lambda^{-1})} \rho(\alpha)$.
As for the factor $f_h x^{e(h)} S_{v w(\lambda)}^Y 1$ 
in the final line of \eqref{eq:EUS1}, denoting $z:=(v w(\lambda))^{-1}$ and 
using Proposition \ref{prp:xzSY} and Corollary \ref{cor:xmuE}, we have 
\begin{align}\label{eq:factor}
 f_h x^{e(h)} S_{v w(\lambda)}^Y 1
=f_h \sum_{q \in \Gamma(\oa{e(h)}^{-1},z)}
 S_{e(q)^{-1}}^Y n_q(Y) g_q (Y) 1 
=f_h \sum_{q \in \Gamma^C(\oa{e(h)}^{-1},z)} n_q g_q E_{\varpi(q)}(x).
\end{align}
We will rewrite this sum over uncolored alcove walks in $\Gamma^C(\oa{e(h)}^{-1},z)$
as a sum over colored alcove walks in  $\Gamma^C_2(\oa{w(\mu)}^{-1},z)$.

Let us given an uncolored alcove walk $q \in \Gamma^C(\oa{e(h)}^{-1},z)$.
Since $q$ is an alcove walk of type $\oa{e(h)}^{-1}=\oa{w(\mu)}^{-1}$,
we can compare the bit sequence of $q$ with the bit sequence of $h$.
In this comparison, if the $k$-th step of $q$ is a folding and 
the $k$-th step of $h$ is a crossing, then we color the $k$-th folding step of $q$ by gray.
Otherwise we color it by black.
Thus we obtain a colored alcove walk, which is denoted by $p$.
Note that we have $p \in \Gamma^C_2(\oa{w(\mu)}^{-1},z)$.
Then each term of the right hand side in \eqref{eq:factor} is equal to 
\[
 f_h n_q g_q E_{\varpi(q)}(x) = f_{p^*} n_{p} g_{p} E_{\varpi(p)}(x),
\]
where $p^*$ is given by \eqref{eq:p*}.
We can also express $f_{p^*}$ using $\beta_k=s_{i_1}\cdots s_{i_{k-1}}(\alpha_{i_k})$ as
\[
 f_{p^*} =
 \prod_{k\in\varphi_+(p^*)}\psi_{i_k}^{+}(q^{\sh(-\beta_k)} t^{\hgt(-\beta_k)})
 \prod_{k\in\varphi_-(p^*)}\psi_{i_k}^{-}(q^{\sh(-\beta_k)} t^{\hgt(-\beta_k)}).
\]
As a result, the last line of \eqref{eq:EUS1} is rewritten by 
a sum over $p \in \Gamma^C_2(\oa{w(\mu)}^{-1},z)$.

Divided by the factor $t_{w_\lambda}^{-\hf}W_\lambda(t)$,
the left hand side of \eqref{eq:EUS1} 
is equal to $P_\lambda(x)$.
Thus we have
\begin{align*}
  E_\mu(x)P_\lambda(x)
&=\sum_{v \in W^\lambda} A_p 
  \sum_{h \in \Gamma(\oa{w(\mu)},e)} f_h 
  \sum_{q \in \Gamma^C(\oa{e(h)}^{-1},z)} n_q g_q E_{\varpi(q)}(x)\\
&=\sum_{v \in W^\lambda} A_p 
  \sum_{p \in \Gamma^C_2(\oa{w(\mu)}^{-1},z)} f_{p^*} g_p n_p E_{\varpi(p)}(x).
\end{align*}
We obtain the result by collecting the terms from $f_{p^*}$, $g_p$ and $n_p$
which depend only on the $k$-th step of $p \in \Gamma^C_2(\oa{w(\mu)}^{-1},z)$
and denoting them by $C_{p,k}$.
\end{proof}

\subsection{Littlewood-Richardson coefficents for Koornwinder polynomials}

In this subsection, we derive our main Theorem \ref{thm:LR}
on LR coefficients of Koornwinder polynomials.

We start with a preliminary lemma.
Recall the complete system $W^\lambda$ of representatives of $W_0/W_\lambda$ 
in \eqref{eq:W^mu} and the element $w(\lambda) \in t(\lambda)W_0$ in \eqref{eq:w(mu)}.

\begin{lem}[{c.f.\ \cite[Proposition 3.7]{Y}}]\label{lem:USS1}
Let $\lambda \in (\frh_\bZ^*)_+$. 
If $v \in W^\lambda$ satisfies $v w(\lambda) \bgeq w(\lambda)$, then we have
\[
 U S_v^Y S_{w(\lambda)}^Y 1 = 
 \Bigl[ \prod_{\alpha\in \cL(w(\lambda)^{-1}, (v w(\lambda))^{-1})} \rho(-\alpha)\Bigr]
 U S_{w(\lambda)}^Y 1,
\]
where $\rho(\alpha)$ is defined in Lemma \ref{lem:PmuEvmu}.
\end{lem}

\begin{proof}
Recall the equality $U T_i=U t_i^\hf$ for $i=1,\dots,n$ in \eqref{eq:UTUt}.
Therefore
we have
\[
 U S_i^Y S_w^Y 1
=U \bigl(T_i^\vee+\psi_i^+(Y^{-\alpha_i})\bigr) S_w^Y 1
=U S_w^Y \bigl(t_i^\hf+\psi_i^+(q^{\sh(-w^{-1}\alpha_i)} t^{\hgt(-w^{-1}\alpha_i)})\bigr)1.
\]
Assume that $v \in W^\lambda$ satisfies $v w(\lambda) \bgeq w(\lambda)$, 
and take a reduced expression $v=s_{i_1}\cdots s_{i_r}$.
Using the above relation, we expand the product 
$U S_v^Y S_{w(\lambda)}^Y 1 = U S_{i_1}^Y\cdots S_{i_r}^Y S_{w(\lambda)}^Y 1$ in order.
We have
\begin{align*}
U S_v^Y S_{w(\lambda)}^Y 1
&=U (t_{i_1}^\hf+\psi_{i_1}^+(Y^{-\alpha_{i_1}})) S_{s_{i_1}v}^Y S_{w(\lambda)}^Y 1
 =U S_{s_{i_1}v}^Y 
  (t_{i_1}^\hf+\psi_{i_1}^+(Y^{-s_{i_r}\cdots s_{i_2}(\alpha_{i_1})}))
  S_{w(\lambda)}^Y 1 \\
&=(t_{i_1}^\hf+\psi_{i_1}^+(Y^{-s_{i_r}\cdots s_{i_2}(\alpha_{i_1})})) U 
  (t_{i_2}^\hf+\psi_{i_2}^+(Y^{-\alpha_{i_2}})) S_{s_{i_1}s_{i_2}v}^Y S_{w(\lambda)}^Y 1 \\
&=\cdots \\
&=U \Bigl[\prod_{j=1}^r \bigl(
   t_{i_j}^\hf+\psi_{i_j}^+(Y^{-w(\lambda)^{-1}s_{i_r}\cdots s_{i_{j+1}}\alpha_{i_j}})
  \bigr) \Bigr] S_{w(\lambda)}^Y 1 \\
&=U S_{w(\lambda)}^Y \Bigl[\prod_{j=1}^r \bigl(
   t_{i_j}^\hf+\psi_{i_j}^+(Y^{-w(\lambda)^{-1}s_{i_r}\cdots s_{i_{j+1}}\alpha_{i_j}})
  \bigr) \Bigr]1 
 =U S_{w(\lambda)}^Y 
 \Bigl[\prod_{\alpha \in \cL(w(\lambda)^{-1}, (v .w(\lambda))^{-1})} \rho(-\alpha)\Bigr]1.
\end{align*}
Therefore the claim is obtained.
\end{proof} 
 
We prepare some symbols to state the main theorem.
For $\mu \in \frh_\bZ^*$, the orbit $W_0.\mu$ contains a unique dominant weight.
We denote it by
\begin{align}\label{eq:dominant}
 \mu_+ \in W_0.\mu \cap (\frh_\bZ^*)_+.
\end{align} 
Let us also recall the set $\Gamma^C_2(\oa{v},w)$ of 
colored alcove walks defined in \eqref{eq:GC2}.

\begin{thm}\label{thm:LR}
Let us given dominant weights $\lambda,\mu \in (\frh_\bZ^*)_+$.
Choose a reduced expression $w(\lambda)=s_{i_r}\cdots s_{i_1}$ of 
the shortest element $w(\lambda) \in t(\lambda)W_0$ in \eqref{eq:w(mu)}.
Then we have
\[
 P_\lambda(x) P_\mu(x)=
 \frac{1}{t_{w_\lambda}^{-\hf}W_\lambda(t)}\sum_{v\in W^\mu} 
 \sum_{p\in\Gamma^C_2(\oa{w(\lambda)}^{-1},(v w(\mu))^{-1})} 
 A_p B_p C_p P_{-w_0.\wtt(p)}(x).
\]
Here $A_p := 
\prod_{\alpha \in w(\mu)^{-1}\cL(v^{-1},v_\mu^{-1})} \rho(\alpha)$
with $\rho(\alpha)$ given in Proposition \ref{prp:EPE}.
The term $C_p$ is the same as that in Proposition \ref{prp:EPE},
and $\wgt(p) \in \frh_\bZ^*$ is defined by \eqref{eq:wtdr}.
The term $B_p$ is defined by 
\[
 B_p := \prod_{\alpha\in \cL(t(\wgt(p))w_0, e(p))} \rho(-\alpha).
\]
\end{thm}

\begin{proof}
The strategy is to calculate the product of Koornwinder polynomials 
by acting the symmetrizer $U$ to each side of the equation in Proposition \ref{prp:EPE}.

For a colored alcove walk 
$p \in \Gamma^C_2(\oa{w(\lambda)}^{-1},(v w(\mu))^{-1})$,
let $z \in W_0$ be the shortest element among $\{z \in W_0 \mid z.\varpi(p)_+=\varpi(p)\}$.
Note that we have $w(\varpi(p)_+)^{-1}=w(-w_0\varpi(p)_+)$.
Since $e(p) \in t(\wgt(p)) W_0$ by the definition of $\wgt(p)$,
$w(-w_0\varpi(p)_+)$ is the shortest element among $t(\wgt(p)) W_0$.
By Lemma \ref{lem:USS1}, we then have
\begin{align*}
  U E_{\varpi(p)}(x) 1
=U S_z^Y S_{w(\varpi(p)_+)}^Y 1 
=\Bigl[ \prod_{\alpha \in \cL(t(\wgt(p))w_0, e(p))} \rho(-\alpha)\Bigr]
 P_{-w_0.\wtt(p)}(x).
\end{align*}
By Proposition \ref{prp:EPE} and this equality, we have
\begin{align*}
P_\lambda(x)P_\mu(x)
&=\frac{1}{t_{w_\lambda}^{-\hf}W_\lambda(t)} U E_\lambda(x) P_\mu(x) 1 \\
&=\frac{1}{t_{w_\lambda}^{-\hf}W_\lambda(t)} \sum_{v \in W^\mu}  
  \sum_{p\in\Gamma^{C}_2(\oa{w(\lambda)}^{-1},(v w(\mu))^{-1})} 
A_p  B_p C_p P_{-w_0.\wtt(p)}(x).
\end{align*}
Hence the claim is obtained.
\end{proof} 

\section{Special cases of Littlewood-Richardson coefficients}\label{s:special}

In Theorem \ref{thm:LR}, we derived an explicit formula of
the LR coefficient $c_{\lambda,\mu}^\nu$ in the product 
$P_\lambda(x)P_\mu(x)=\sum_{\nu}c_{\lambda,\mu}^\nu P_\nu(x)$ 
of Koornwinder polynomials using alcove walks.
In this section, we discuss several specializations of the formula.

\subsection{Askey-Wilson polynomials}\label{ss:sp:AW}

As mentioned in \S \ref{ss:intro:K},
Koornwinder polynomials in the rank one case are nothing but Askey-Wilson polynomials.
In this case LR coefficients of Askey-Wilson polynomials are 
expected to be simpler than the general rank case in Theorem \ref{thm:LR}.

As a preparation, we summarize the data of the root system of rank $1$.
We consider the Euclid space $V=\bR \ep^\vee$ of dimension $1$ 
and its dual space $V^*=\bR \ep$.
The root system of type $C_1$ is $R=\pr{\pm 2\ep} \subset V^*$,
the simple root is $\alpha_1=2\ep$, and the fundamental weight is $\omega=\ep$.
The weight lattice is $P = \frh_\bZ^* = \bZ\ep \subset V^*$, 
and the set of dominant weights is $(\frh_\bZ^*)_+ = \bN \ep$.
The finite Weyl group $W_0$ is the group of order two generated by $s_1:=s_{\alpha_1}$, 
and the longest element of $W_0$ is $w_0=s_1$.
The affine root system of rank $1$ is 
$S=\pr{\pm2\ep+k\delta,\pm\ep+\frac{k}{2}\delta \mid k\in\bZ}$ with $\alpha_0=\delta-2\ep$, and 
the extended affine Weyl group $W$ is the group generated by $s_1$ and $s_0:=s_{\alpha_0}$.
The decomposition $W=t(P) \rtimes W_0$ \eqref{eq:exafW} is 
the semi-direct product of $t(P)=\langle t(\ep_1)=s_0 s_1 \rangle \simeq \bZ^2$ and
$W_0=\langle s_1 \rangle \simeq \bZ/2\bZ$.

We denote by
\[
 P_l(x) = P_l(x; q,t_0,t_1,u_0,u_1)
\]
the Askey-Wilson polynomial associated to the dominant weight 
$\lambda=l \omega=l \ep$ ($l \in \bN$).
Note that it has five parameters.

First, we consider the simplest case.  
Following the case of type $A$ (see \S \ref{ss:intro:LR}), 
we call the LR coefficients $c_{\lambda,\mu}^\nu$ 
with $\lambda$ or $\mu$ equal to a minuscule weight \emph{Pieri coefficients}.
Since the weight $\omega_1$ is the unique minuscule weight in the root system of type $C_n$,
we consider the case $\lambda=\omega$ for the rank one case.
 
Let us write down explicitly the Askey-Wilson polynomial $P_1(x)=P_\omega(x)$. 
In the following calculation, we need an explicit form of the term
$\rho(\alpha)$ ($\alpha \in S$) in Proposition \ref{prp:EPE} and Theorem \ref{thm:LR}.
The result is:
\begin{align}\label{eq:rho:n=1} 
\rho(\alpha) := 
 t_1^\hf
 \frac{(1+q^{\frac{k}{2}} t_0^{ \hf}t_1^{-\hf} (t_0 t_1)^{-\frac{j}{2}})
       (1-q^{\frac{k}{2}} t_0^{-\hf}t_1^{-\hf} (t_0 t_1)^{-\frac{j}{2}})}
      {1-q^{k}(t_0 t_1)^{-j}} 
 \quad (\alpha = 2j \ep +k \delta \in S).
\end{align}

\begin{lem}
The Askey-Wilson polynomial associated to the minuscule weight $\omega$ is
\[
 P_1(x) = x+x^{-1} 
 +\rho(2\delta-\alpha_1)\psi_0^-(qt_0t_1)+t_1^\hf\psi_0^+(qt_0t_1)
 +\psi_1^+(q^2t_0t_1)\psi_0^-(qt_0t_1).
\]
Here $\psi_k^{\pm}(z)$ ($k=0,1$) is given by \eqref{eq:psi(z)} with $n=1$.
Explicitly, we have
\begin{align}\label{eq:psi:n=1}
 \psi_0^{\pm}(z) := 
 \mp\frac{(u_1^\hf-u_1^{-\hf})+z^{\pm\hf}(u_0^\hf-u_0^{-\hf})}{1-z^{\pm1}},
 \quad 
 \psi_1^{\pm}(z) :=
 \mp\frac{(t_1^\hf-t_1^{-\hf})+z^{\pm\hf}(t_0^\hf-t_0^{-\hf})}{1-z^{\pm1}}.
\end{align}
\end{lem}

\begin{proof}
Below we use the word \emph{non-symmetric Askey-Wilson polynomials} to mean 
non-symmetric Koornwinder polynomials (Fact \ref{fct:Emu}) in the rank $1$ case.
By Lemma \ref{lem:PmuEvmu}, we can rewrite $P_1(x)$ as a linear combination 
of non-symmetric Askey-Wilson polynomials $E_k(x)=E_{k\omega}(x)$, $k\in\bZ$.
The result is
\[
 P_1(x)=\rho(2\delta-\alpha_1)E_1(x)+E_{-1}(x).
\]
Next, using the Ram-Yip type formula (Fact \ref{fct:RY}), 
we can expand $E_1(x)$ and $E_{-1}(x)$ by monomials.
The results are
\[
 E_1(x) = t_1^\hf x+\psi^-_0(qt_0t_1), \quad 
 E_{-1}(x) = x^{-1}+t_1^\hf\psi_0^+(qt_0t_1)+t_1^\hf\psi^+_1(q^2t_0t_1)x
 +\psi_1^+(q^2t_0t_1)\psi_0^-(qt_0t_1).
\]
By these formulas, we have
\[
 P_1(x)= x^{-1}+(t_1^\hf\rho(2\delta-\alpha_1)+t_1^\hf\psi^+_1(q^2t_0t_1))x
+\psi^-_0(qt_0t_1)\rho(2\delta-\alpha_1)
+t_1^\hf\psi_0^+(qt_0t_1)
+\psi_1^+(q^2t_0t_1)\psi_0^-(qt_0t_1).
\]
By a direct calculation, the coefficients of $x$ is shown to be 
\[
 t_1^\hf\rho(2\delta-\alpha_1)+t_1^\hf\psi^+_1(q^2t_0t_1)=1.
\]
Therefore the claim is obtained.
\end{proof}

\begin{rmk}\label{rmk:AW}
Let us replace the parameters $(q,t_0,t_1,u_0,u_1)$ with 
the original parameters $(q,a,b,c,d)$ of Askey-Wilson polynomials in \cite{AW}.
The correspondence \eqref{eq:AW-K} of parameters can be rewritten as
\[
 (q,t_0,t_1,u_0,u_1) = (q,-q^{-1}a b, -c d, -a/b, -c/d).
\]
Using this correspondence and the relation $a b c d = q t_0 t_n$, 
we can rewrite $P_1(x)$ as 
 \[
 P_1(x) = x+x^{-1}+\frac{\pi s - s'}{1- \pi}, \quad 
 \pi := a b c d, \ s := a+b+c+d, \ s' := a^{-1}+b^{-1}+c^{-1}+d^{-1}.
\]
We can then compare $P_1(x)$ with the original Askey-Wilson polynomials $p_n(z)$ in \cite[p.5]{AW}.
By loc.\ cit., we have $p_1(z)= 2(1-\pi)z+\pi s-s'$, and thus
\[
 (1-\pi) P_1(x) = p_1\bigl((x+x^{-1})/2\bigr).
\]
Therefore they coincide up to the normalization factor.
\end{rmk}

\begin{prp}\label{prp:AW_P}
For a dominant weight $\lambda=l \omega \in(\frh_\bZ^*)_+$, $l\in\bN$, we have
\begin{align*}
 P_1(x)P_l(x) &=  P_{l+1}(x) +F_l P_l(x)+G_l P_{l-1}(x), \\
 F_l &:= \rho(-2l\delta+\alpha_1)(-\psi_0^-(q^{2l+1}t_0t_1)+\psi_0^-(q t_0t_1))+\rho(2l\delta-\alpha_1)
         (-\psi_0^+(q^{2l-1}t_0t_1)+\psi_0^+(q t_0t_1)), \\
 G_l &:= \rho(2l\delta-\alpha_1)\rho(-2(l-1)\delta+\alpha_1)n_0(q^{2l-1}t_0t_1).
\end{align*}
Here $\rho(\alpha)$ is given by \eqref{eq:rho:n=1},
$\psi^{\pm}_0(z)$ is given by \eqref{eq:psi:n=1}, and 
$n_0(z)$ is given in Proposition \ref{prp:SYiw} with $n=1$.
Explicitly, the last one is given by
\begin{align*}
&n_0(z) :=
 \frac{(1-u_1^{ \hf}u_0^\hf z^\hf)(1+u_1^{ \hf}u_0^{-\hf}z^\hf)}{1-z}
 \frac{(1+u_1^{-\hf}u_0^\hf z^\hf)(1-u_1^{-\hf}u_0^{-\hf}z^\hf)}{1-z}.
\end{align*}
\end{prp}

\begin{proof}
By Theorem \ref{thm:LR}, we have
\[
 P_\lambda(x) P_\mu(x) =
 \frac{1}{t_{w_\lambda}^{-\hf}W_\lambda(t)}\sum_{v\in W^\mu} 
 \sum_{p\in\Gamma^C_2(\oa{w(\lambda)}^{-1},(v w(\mu))^{-1})} 
 A_p B_p C_p P_{-w_0.\wtt(p)}(x)
\]
for dominant weights $\lambda,\mu \in (\frh_\bZ^*)_+$. 
We apply this equation to the case $\lambda=\omega$ and $\mu=l\omega$.
In this case the stabilizer $W_\mu \subset W_0$ in \eqref{eq:W_mu} 
is $W_\mu=\pr{e}$, and the complete system $W^\mu$ \eqref{eq:W^mu} 
of representatives of $W_0/W_\lambda$ is $W^{l\omega}=\pr{e,s_1}$.
As for the shortest element $w(\nu) \in t(\nu)W_0$ given in \eqref{eq:w(mu)},
we have by $t(\omega)=s_0s_1$ that $w(\omega)=s_0$ and $w(l\omega)=(s_0s_1)^{l-1}s_0$.
 
First, we calculate the denominator $t_{w_\omega}^{-\hf}W_\omega(t)$.
As for the longest element $w_\lambda \in W^\lambda$ in \eqref{eq:wmu}, 
we have $w_\omega=e$.
Thus, by recalling the definition \eqref{eq:t_w} of $t_w$ ($w \in W$), we have
 $t_{w_\omega}^{-\hf}W_\omega(t)= t_{e}^{-\hf} t_e= 1$.

Next, as for the sum in the right hand side, we calculate the case $v=s_1$.
The set of alcove walks is then
$\Gamma_2^C(\oa{w(\lambda)}^{-1},(v w(\mu))^{-1})=\Gamma_2^C(\oa{s_0},t(l\omega))$.
In the upper half of Table \ref{tab:AW_P}, we display the alcove walks $p$ therein 
together with the corresponding terms $A_p$, $B_p$ and $C_p$. 
In the table we denote by $H_0$ and $H_1$ 
the hyperplanes in the $W$-orbits of $H_{\alpha_0}$ and $H_{\alpha_1}$ respectively.
We also denote a black folding by a solid line, 
and a gray folding by a dotted line.

Next we study the case $v=e$.
The set of alcove walks is 
$\Gamma_2^C(\oa{w(\lambda)}^{-1},(v w(\mu))^{-1})=\Gamma_2^C(\oa{s_0},w(l\omega))$,
and in the lower half of Table \ref{tab:AW_P} we display the alcove walks $p$ therein 
together with the corresponding terms $A_p$, $B_p$ and $C_p$.

\begin{table}[htbp]
\centering
\includegraphics[bb= 0 0 445 352]{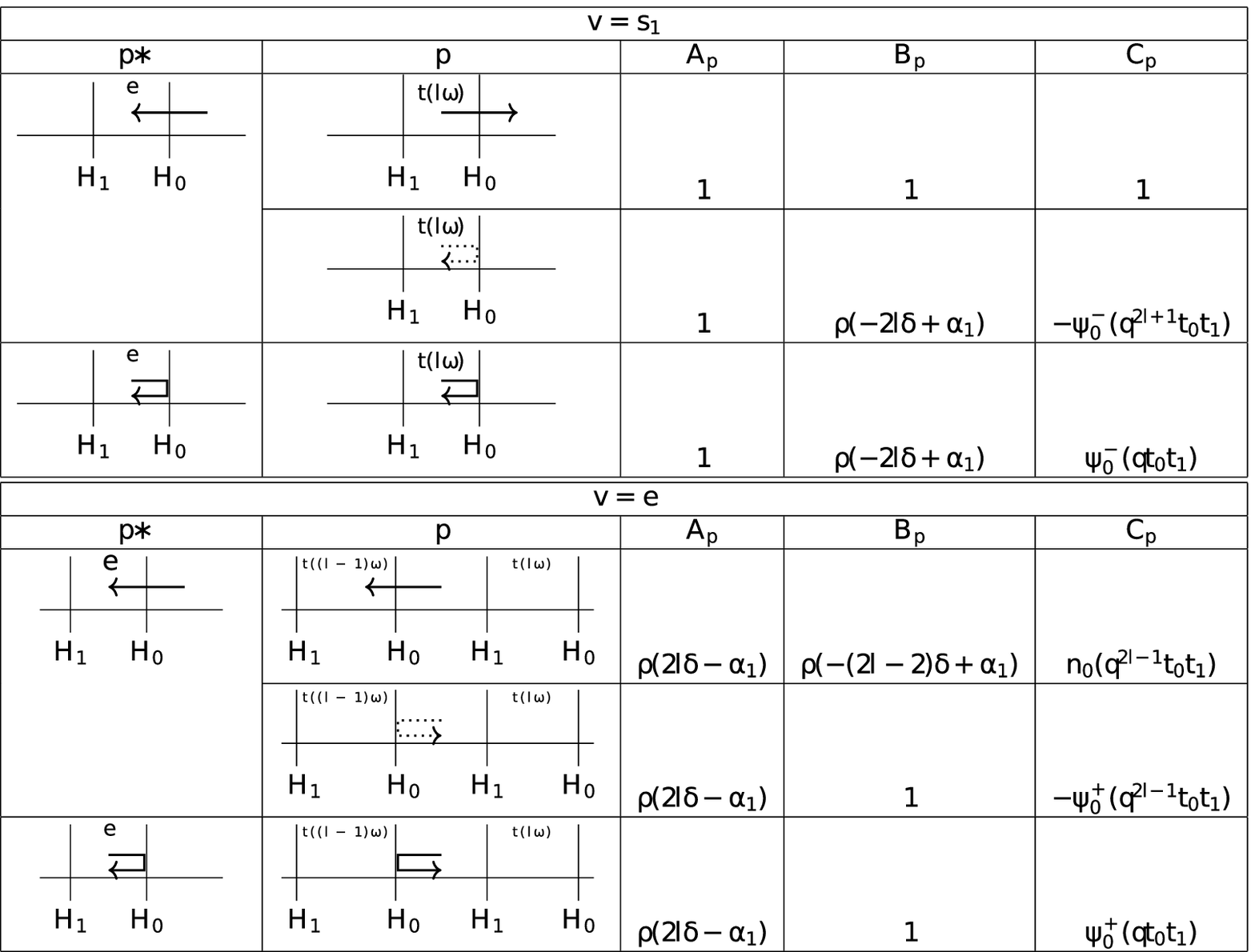}
\caption{Colored alcove walks in Proposition \ref{prp:AW_P}}
\label{tab:AW_P}
\end{table}
The claim is obtained by to summaries the above calculation.
\end{proof}

\begin{rmk}\label{rmk:AW2}
Continuing Remark \ref{rmk:AW}, 
we rewrite the result in Proposition \ref{prp:AW_P}
in terms of the original parameters $(q,a,b,c,d)$ of Askey-Wilson polynomials.
The result is
\begin{align}\label{eq:AW2:rec}
 P_1(x)P_l(x) = P_{l+1}(x) + F_l P_l(x) + G_{l-1} P_{l-1}(x),
\end{align}
where the factors $F_l$ and $G_l$ are given by 
\begin{align*}
&F_l := \frac{f_l+(\pi s'-s)}{1-\pi}, \quad 
 f_l := q^{l-1}\frac{(1+q^{2l-1} \pi)(q s+\pi s')-q^{l-1}(1+q)\pi(s+q s')}
               {(1-q^{2l-2}\pi)(1-q^{2l}\pi)}, \\
&G_l := \frac{g_l \gamma_{l-1}}{\gamma_1 \gamma_l}, \quad 
 g_l := (1-q^l) \frac{(1-q^{l-1} a b)(1-q^{l-1} a c)(1-q^{l-1} a d)
                      (1-q^{l-1} b c)(1-q^{l-1} b d)(1-q^{l-1} c d)}
                     {(1-q^{2l-2} \pi)(1-q^{2l-1} \pi)}, \\
&\pi:=a b c d, \quad s:= a+b+c+d, \quad s':=a^{-1}+b^{-1}+c^{-1}+d^{-1}, \\ 
&\gamma_l:=(q^{l-1} \pi;q)_l = (1-q^{l-1} \pi)(1-q^l \pi) \cdots (1-q^{2l-2} \pi).
\end{align*}
In the case $l=0$, we have $\rho(-\alpha_1)=0$, and thus $F_1=0$.
If we define $p_l(z)$ by the relation $P_l(x) = \gamma_l^{-1} p_l((x+x^{-1})/2)$,
then the relation \eqref{eq:AW2:rec} can be rewritten as 
\[
 2 z p_l(z) = h_l p_{l+1}(z) + f_l p_l(z) + g_l p_{l-1}(z), \quad 
 h_l := \frac{1-q^{l-1}\pi}{(1-q^{2l-1}\pi)(1-q^{2l}\pi)}, \quad
 p_0(z)=1, \ p_{-1}(z)=0.
\]
This recurrence formula is nothing but the one in \cite[(1.24)--(1.27)]{AW}.
Thus $p_l$ coincides with the original Askey-Wilson polynomial in \cite{AW},
and in particular, it can be expressed as a $q$-hypergeometric series.
\end{rmk}

So far we studied Pieri coefficients.
Next we study the general LR coefficients for Askey-Wilson polynomials. 

\begin{cor}\label{cor:AW_LW}
For dominant weights $l\omega$ and $m \omega$ in $(\frh^*_\bZ)_+$, 
$l,m \in\bN$, we have 
\[
 P_{l \omega}(x) P_{m \omega}(x) = 
 \sum_{v\in W_0} \sum_{p \in\Gamma^C_2(\oa{t((l-1)\omega)s_0},t(m\omega)s_1v)}
 A^{AW}_p B_p^{AW} C_p P_{\wgt(p)}(x),
\]
where the terms $A^{AW}_p$ and $B^{AW}_p$ are given by 
\begin{align*}
 A^{AW}_p:=
 \begin{dcases}
  \rho(2m\delta-\alpha_1) & (v=e) \\ 
  1 & (v=s_1) 
 \end{dcases},\quad
 B^{AW}_p:=
 \begin{dcases}
  \rho(-\ell(e(p))\delta+\alpha_1) & (\ell(e(p)) \in 2\bZ)\\
  1 & (\ell(e(p)) \notin 2\bZ)
 \end{dcases}
\end{align*}
with $\rho(\alpha)$ in Proposition \ref{prp:AW_P},
and $C_p$ is given in Theorem \ref{thm:LR}.
\end{cor}

\begin{proof}
We apply the formula 
\[
 P_\lambda(x) P_\mu(x) =
 \frac{1}{t_{w_\lambda}^{-\hf}W_\lambda(t)}\sum_{v\in W^\mu} 
 \sum_{p \in \Gamma^C_2(\oa{w(\lambda)}^{-1},(v w(\mu))^{-1})} 
 A_p B_p C_p P_{-w_0.\wtt(p)}(x)
\]
in Theorem \ref{thm:LR} to the case $\lambda=l\omega$ and $\mu=m\omega$.
Similarly as in Proposition \ref{prp:AW_P}, 
we have $W_{l\omega}=\pr{e}$ and $W^{m\omega}=\pr{e,s_1}=W_0$.
Using $t(l \omega)=(s_0s_1)^l$ and $t(m \omega)=(s_0s_1)^m$, we have
$w(l \omega)=t(l \omega)s_1=(s_0s_1)^{l-1} s_0$ and 
$w(m \omega)=(s_0s_1)^{m-1}s_0$.
Therefore the range of the sum of alcove walks in the right hand side becomes
\[
 \Gamma^C_2(\oa{w(\lambda)}^{-1},(v w(\mu))^{-1})=
 \Gamma^C_2(\oa{t((l-1)\omega)s_0}, t(m\omega)s_1v) \quad (v\in W_0).
\]
As for the denominator $t_{w_{l \omega}}^{-\hf}W_{l \omega}(t)$, we have 
by $w_{l \omega}=e$ that $t_{w_{l \omega}}^{-\hf}W_{l \omega}(t)= t_{e}^{-\hf} t_e= 1$.

Now we study the factors $A_p$ and $B_p$, and want to reduce the ranges of the products.
First, as for the product 
$A_p=\prod_{\alpha\in\cL((vw(\mu))^{-1},t(-w_0\mu))}\rho(\alpha)$, 
the longest element $w_0 \in W_0$ is $s_1$ and $t(\mu)=(s_0 s_1)^l=w(\mu)^{-1}s_0$.
Thus, in the case $v=e$, we have 
\[
 \cL((vw(\mu))^{-1},t(-w_0\mu))
=\cL(w(\mu)^{-1},t(\mu))
=\pr{2m\delta-\alpha_1}.
\]
In the case $v=s_1$, we have 
\[
 \cL((vw(\mu))^{-1},t(-w_0\mu))
=\cL(w(\mu)s_1,t(\mu))
=\cL(t(\mu),t(\mu))
=\emptyset.
\]
Hence $A_p$ is equal to $A_p^{AW}$ in the claim.

Next we consider the product 
$B_p = \prod_{\alpha\in \cL(t(\wgt(p))w_0, e(p))} \rho(-\alpha)$.
we separate the argument according to whether 
the length $\ell(e(p))$ of $e(p)$ is even or odd.
In the case $\ell(e(p))$ is even, there is $k \in \bN$ such that 
$e(p)=(s_0 s_1)^{n}=t(k\omega)$, $0 \le k \le m$.
In this case, the range of the product is
\[
 \cL(t(\wgt(p))w_0, e(p))=\cL(t(k\omega)s_1,t(k\omega))=\pr{2k\delta-\alpha_1}.
\]
In the case $\ell(e(p))$ is odd, there is $k \in \bN$ such that 
we can write $e(p)=(s_0 s_1)^{k-1}s_0=t(k\omega)s_1 $, $1 \le k \le m$.
Thus the range of the product is
\[
 \cL(t(\wgt(p))w_0, e(p))=\cL(t(k\omega)s_1,t(k\omega)s_1)=\emptyset.
\]
Therefore $B_p$ is equal to $B_p^{AW}$ in the claim.
\end{proof}

\subsection{Hall-Littlewood limit}\label{ss:sp:HL}

In the case of type $A_n$, the specialized Macdonald polynomials 
$P^{A_n}_\lambda(x;q=0,t)$ coincide with Hall-Littlewood polynomials.
Motivated by this fact, Yip calls in \cite[\S 4.5]{Y} the specialized Macdonald polynomials 
in the untwisted cases at $q=0$ \emph{Hall-Littlewood polynomials}, 
and derived a simplified formula of LR coefficients.  
Following Yip's terminology, let us call the specialized Koornwinder polynomials
\[
 P_\lambda(x;t) := P_\lambda(x;q=0,t_0,t,t_n,u_0,u_n)
\]
\emph{the Hall-Littlewood limit}.

\begin{prp}[{c.f.\ \cite[Corollary 4.13]{Y}}]\label{prp:HL}
Let us given dominant weights $\lambda,\mu \in (\frh_\bZ^*)_+$ and a reduced expression
$w(\lambda)=s_{i_r}\cdots s_{i_1}$ of the shortest element $w(\lambda)$ \eqref{eq:w(mu)}.
Then we have
\begin{align*}
   P_\lambda(x;t) P_\mu(x;t)
&= \frac{1}{t_{w_\lambda}^{-\hf}W_\lambda(t)}\sum_{v\in W^\mu} 
   \sum_{p\in\Gamma^C_+(\oa{w(\lambda)}^{-1},(v w(\mu))^{-1})} 
   F_p(t)  P_{-w_0.\wtt(p)}(x;t), \\
   F_p(t) 
&:=\prod_{\alpha\in\cL((vw(\mu))^{-1},t(-w_0\mu))} t_\alpha^\hf
   \prod_{\alpha\in\cL(t(\wgt(p))w_0,e(p))} t_\alpha^{-\hf} \\
&  \quad \times
   \prod_{k\in\varphi_+(p), \ \alpha_{i_k}\not\in W.\alpha_0} 
   (t_{\alpha_{i_k}}^{-\hf}-t_{\alpha_{i_k}}^\hf) 
   \prod_{k\in\varphi_+(p), \ \alpha_{i_k}\in W.\alpha_0}
   (u_n^{-\hf}-u_n^\hf) .
\end{align*}
Here $\Gamma^C_+(\oa{w(\lambda)}^{-1},(v w(\mu))^{-1})$ is the subset
of $\Gamma^C(\oa{w(\lambda)}^{-1},(v w(\mu))^{-1})$ 
consisting of alcove walks whose foldings are positive.
\end{prp}

\begin{proof}
We denote the coefficient in Theorem \ref{thm:LR} by 
\[
 a_p(q,t) := A_p B_p C_p. 
\]

First, we show that if $a_p(0,t)\neq0$ for a colored alcove walk 
$p \in \Gamma_2^C(\oa{w(\lambda)}^{-1},(v.w(\mu))^{-1})$,
then all the foldings of $p$ are gray and positive.
We assume that the $k$-th step of $p$ is a gray negative folding.
Then, as for the factor $C_{p,k}=-\psi^-_{i_k}(q^{\sh(-h_k(p))}t^{\hgt(-h_k(p))})$
we have $\rst{C_{p,k}}{q=0}=0$.
In fact, we have 
\[
 \psi_{i_k}^-(z)
=\frac{(t_{i_k}^\hf-t_{i_k}^{-\hf})+z^{-\hf}(u_{i_k}^\hf-u_{i_k}^{-\hf})}{1-z^{-1}}
=\frac{z(t_{i_k}^\hf-t_{i_k}^{-\hf})+z^{\hf}(u_{i_k}^\hf-u_{i_k}^{-\hf})}{1-z},
\]
and by substituting $z=q^{\sh(-h_k(p))}t^{\hgt(-h_k(p))}$ and $q=0$ 
we have $\rst{C_{p,k}}{q=0}=0$.
Thus we showed that no gray negative folding contributes to $a_p(0,t)$.

Next we show that black foldings of $p$ don't contribute to $a_p(0,t)$.
Note that there exists an alcove walk $l \in \Gamma(\oa{w(\lambda)},e)$ 
whose steps are crossings since we fixed a reduced expression of $w(\lambda)$.
Moreover all the steps of $l$ are positive.
Then we find that any alcove walk in $\Gamma_2^C(\oa{w(\lambda)},e) \setminus \{l\}$
has a negative folding. 
In other words, if an alcove walk 
$p \in \Gamma_2^C(\oa{w(\lambda)}^{-1},(v w(\mu))^{-1})$ has a black folding, 
then $p^*$ in \eqref{eq:p*} has at least one negative folding.
Then, as for the factor $C_{p,k}=-\psi^-_{i_k}(q^{\sh(-\beta_k)}t^{\hgt(-\beta_k)})$,
we have $\rst{C_{p,k}}{q=0}=0$ by a direct calculation.
Thus, no black folding contributes to $a_p(0,t)$.

By the discussion so far, we find that neither colored folding contributes to $a_p(0,t)$.
Thus, the set of alcove walks effective to the sum is 
$\bigl\{p\in\Gamma^C(\oa{w(\mu)}^{-1},(v.w(\lambda))^{-1}) \mid \varphi(p)=\varphi_+(p)\bigr\}$.
 
Specializing $q=0$ in $A_p$, $B_p$ and $C_p$, we have
\begin{align*}
\rst{A_p}{q=0}
&= \prod_{\alpha\in\cL((vw(\mu))^{-1},t(-w_0\mu))} t_\alpha^\hf,\quad 
\rst{B_p}{q=0}
 = \prod_{\alpha\in\cL(t(\wgt(p))w_0,e(p))} t_\alpha^{-\hf}, \\
\rst{C_p}{q=0}
&= \prod_{k\in\varphi_+(p), \ \alpha_{i_k}\not\in W.\alpha_0} 
   (t_{\alpha_{i_k}}^{-\hf}-t_{\alpha_{i_k}}^\hf) 
   \prod_{k\in\varphi_+(p), \ \alpha_{i_k}\in W.\alpha_0} (u_n^{-\hf}-u_n^\hf).
\end{align*}
Therefore the claim is obtained. 
\end{proof}

\subsection{Examples in rank $2$}\label{ss:rank2}

Finally, as explicit examples of LR coefficients in Theorem \ref{thm:LR}, 
we calculate the product $P_\lambda(x)P_\mu(x)$ of Koornwinder polynomials of rank $2$.

We write down the root system of rank $2$.
{The} root system of type $C_2$ is 
\[
 R:=\pr{\pm\ep_1\pm\ep_2}\cup\pr{\pm2\ep_1,\pm2\ep_2} \subset V^* := \bR\ep_1\oplus \bR\ep_2,
\] 
the simple roots are $\alpha_1=\ep_1-\ep_2$ and $\alpha_2=2\ep_2$,
and the fundamental weights are $\omega_1=\ep_1$ and $\omega_2=\ep_1+\ep_2$.
The weight lattice is $P = \frh_\bZ^* = \bZ\ep_1\oplus \bZ\ep_2 \subset V^*$, and 
the set of dominant weights is $(\frh_\bZ^*)_+ = \pr{\lambda_1\ep_1+\lambda_2\ep_2 \in 
 \frh^*_\bZ \mid \lambda_1\geq\lambda_2\geq0}$.
The finite Weyl group $W_0$ is the hyper-octahedral group of order $8$ generated 
by $s_1:=s_{\alpha_1}$ and $s_2:=s_{\alpha_2}$. 
The longest element of $W_0$ is $w_0=s_1s_2s_1s_2=s_2s_1s_2s_1$.

The affine root system of type $(C_2^\vee,C_2)$ is
\[
 S=\bigl\{ \pm2\ep_i+k\delta, \pm\ep_i+\hf k \delta \mid k \in \bZ, i=1,2 \bigr\}
   \cup\pr{\pm\ep_1\pm\ep_2+k\delta \mid k\in\bZ},
\]
and the affine simple root is $\alpha_0=\delta-2\ep_1$.
The extended affine Weyl group $W$ is generated by $s_1, s_2$ and $s_0:=s_{\alpha_0}$,
and the decomposition $W=t(P) \rtimes W_0$ \eqref{eq:exafW} is a semi-direct product of 
$t(P)=\langle t(\ep_1), t(\ep_2) \rangle \simeq \bZ^2$
and $W_0=\langle s_1,s_2 \rangle \simeq \pr{\pm1}^2\rtimes\frS_2$.
The elements $t(\ep_1)$ and $t(\ep_2)$ have reduced expressions 
$t(\ep_1)=s_0s_1s_2s_1$ and $t(\ep_2)=s_1s_0s_1s_2$ respectively.

In this setting we apply Theorem \ref{thm:LR} to the case $\lambda=\omega_1$ and $\mu=\omega_2$.
The result is as follows.

\begin{prp}\label{prp:rk2}
For Koornwinder polynomials of rank $2$, we have
\begin{align*}
 P_{\omega_1}(x)P_{\omega_2}(x) 
  &= 
    P_{\omega_1+\omega_2}(x) + F P_{\omega_2}(x)  + G P_{\omega_1}(x) , \\
F &:= \rho(-2\delta+(\ep_1+\ep_2))\rho(-2\delta+2\ep_1)\rho(-(\ep_1-\ep_2))(-\psi_0^-(q^3 t_0t_1)+\psi_0^-(q t_0t_1))
\\
G &:=  \rho(2\delta-(\ep_1+\ep_2))\rho(2\delta-2\ep_2)\rho(-2\ep_2)\rho(-\delta+(\ep_1+\ep_2))n_0(qt_0t_1)
\end{align*}
\end{prp}

\begin{proof}
Applying Theorem \ref{thm:LR} to $n=2$, $\lambda=\omega_1$ and $\mu=\omega_2$, we have
\[
 P_{\omega_1}(x)P_{\omega_2}(x)=
 \frac{1}{t^{-\hf}_{w_{\omega_1}} W_{\omega_1}(t)}
 \sum_{v\in W^{\omega_2}}
 \sum_{p\in\Gamma_2^C(\oa{w(\omega_1)}^{-1},(vw(\omega_2))^{-1})}
 A_p B_p C_p P_{-w_0.\wtt(p)}(x).
\]
We have $W_{\omega_1}=\pr{e, s_2}$, $W^{\omega_2}=\pr{e,s_2,s_1s_2,s_2s_1s_2}$ 
and $w(\omega_1)=s_0$, $w(\omega_2)=s_0s_1s_0$.
The denominator $t_{w_{\omega_1}}^{-\hf}W_{\omega_1}(t)$ can be calculated 
with the help of $w_{\omega_1}=s_2$ as 
$t_{w_{\omega_1}}^{-\hf}W_{\omega_1}(t)
 = t_{s_2}^{-\hf} (t_e+t_{s_2}) = t_2^{-\hf} +t_2^\hf$.

Next we consider the term $A_p B_p C_p$.
The alcove walk $p^*$ associated to 
$p \in \Gamma_2^C(\oa{w(\omega_1)}^{-1},(vw(\omega_2))^{-1})$
is given by either $p_1^*$ or $p_2^*$ in Table \ref{tab:2:p*}.
\begin{table}[htbp]
\centering
\includegraphics[bb= 0 0 224 87]{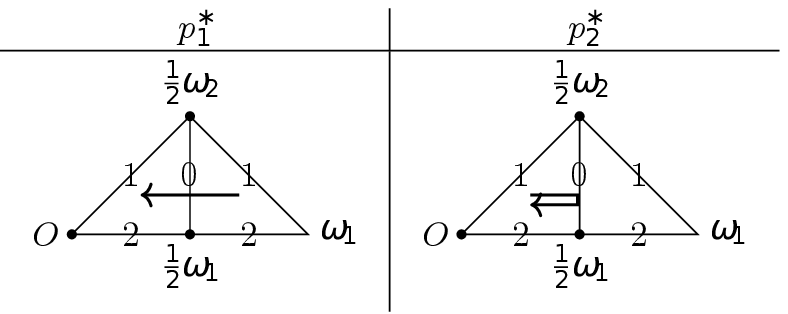}
\caption{Classification of $p^*$}\label{tab:2:p*}
\end{table}

Let us calculate the term $A_p = \prod_{\alpha} \rho(\alpha)$.
The range of the product is
\[
 w(\mu)^{-1} \cL(v^{-1},v_\mu^{-1}) = 
 \cL\bigl((v w(\omega_2))^{-1},t(-w_0 \omega_2) \bigr),
\]
and according to $v \in W^{\omega_2}=\pr{e,s_2,s_1s_2,s_2s_1s_2}$ it is given by 
\[
 \cL\bigl((vw(\omega_2))^{-1},t(-w_0\omega_2)\bigr) =
 \begin{dcases}
  \pr{2\delta-2\ep_1,2\delta-(\ep_1+\ep_2),2\delta-2\ep_2} \quad & (v=e) \\
  \pr{2\delta-(\ep_1+\ep_2),2\delta-2\ep_2} \quad & (v=s_2) \\
  \pr{2\delta-2\ep_2} \quad & (v=s_1s_2) \\
  \emptyset \quad & (v=s_2s_1s_2).
 \end{dcases}
\]
Then we have
\[
 A_p=
 \begin{dcases}
  \rho(2\delta-2\ep_1)\rho(2\delta-(\ep_1+\ep_2))\rho(2\delta-2\ep_2)
  \quad & (v=e) \\
  \rho(2\delta-(\ep_1+\ep_2))\rho(2\delta-2\ep_2) \quad & (v=s_2) \\
  \rho(2\delta-2\ep_2) \quad & (v=s_1s_2) \\
  1 \quad & (v=s_2s_1s_2)
\end{dcases}.
\]

For each $v \in W^{\omega_2}$ and the corresponding colored alcove walks $p$, 
we calculate $B_p$ and $C_p$. 
The results are shown in Tables \ref{tab:2:e}--\ref{tab:2:s_2s_1s_2}.
The symbol in the column of $p$ such as $X_{1 1}$ and $X_{1 2}$ 
refers to the corresponding picture in Figure \ref{fig:4.4}.

\begin{table}[htbp]
\centering
\begin{tabular}{c|c|c|c|c}
  $p^*$ & $p$ & $B_p$ & $C_p$ & $-w_0\wtt(p)$
\\ \hline
$p_1^*$ & $X_{1 1}$ & $\rho(-2\ep_2)\rho(-\delta+(\ep_1+\ep_2))$ & $n_0(qt_0t_1)$ 
        & $\omega_1$
\\ \cline{2-5}
        & $X_{1 2}$ & $\rho(-(\ep_1-\ep_2))$ & $-\psi_0^+(qt_0t_1)$ & $\omega_2$
\\ \hline
$p_2^*$ & $X_2$ & $\rho(-(\ep_1-\ep_2))$ & $\psi_0^+(q t_0t_1)$ & $\omega_2$
\end{tabular}
\caption{Colored lcove walks in the case $v=e$}
\label{tab:2:e}
\end{table}

\begin{table}[htbp]
\centering
\begin{tabular}{c|c|c|c|c}
  $p^*$ & $p$ & $B_p$ & $C_p$ & $-w_0\wtt(p)$ 
\\ \hline
  $p_1^*$ & $Y_{1 1}$ 
& $\rho(-2\ep_2)\rho(-2\delta+2\ep_1) \rho(-\delta+(\ep_1+\ep_2))$ 
& $n_0(q t_0t_1)$ & $\omega_1$
\\ \cline{2-5}
& $Y_{1 2}$ & $\rho(-(\ep_1-\ep_2))\rho(-2\delta+2\ep_1)$ 
& $-\psi_0^+(q t_0t_1)$ & $\omega_2$
\\ \hline
  $p_2^*$ & $Y_2$ & $\rho(-(\ep_1-\ep_2))\rho(-2\delta+2\ep_1)$ 
& $\psi_0^+(q t_0t_1)$ & $\omega_2$
\end{tabular}
%

\caption{Colored alcove walks in the case $v=s_2$}
\label{tab:2:s_2}
\end{table}

\begin{table}[htbp]
\centering
\begin{tabular}{c|c|c|c|c}
$p^*$ & $p$ & $B_p$ & $C_p$ & $-w_0\wtt(p)$ 
\\ \hline
$p_1^*$ & $Z_{1 1}$ & $1$& $1$ & $\omega_1+\omega_2$ 
\\ \cline{2-5}
& $Z_{1 2}$ 
& $\rho(-2\delta+(\ep_1+\ep_2))\rho(-2\delta+2\ep_1)\rho(-(\ep_1-\ep_2))$
& $-\psi_0^-(q^{3}t_0t_1)$ & $\omega_2$ 
\\ \hline
$p_2^*$ & $Z_2$ 
& $\rho(-2\delta+(\ep_1+\ep_2))\rho(-2\delta+2\ep_1)\rho(-(\ep_1-\ep_2))$
& $\psi_0^-(q t_0t_1)$ & $\omega_2$ 
\end{tabular}
\caption{Colored alcove walks in the case $v=s_1 s_2$}
\label{tab:2:s_1s_2}
\end{table}

\begin{table}[htbp]
\centering
\begin{tabular}{c|c|c|c|c}
  $p^*$ & $p$ & $B_p$ & $C_p$ & $-w_0\wtt(p)$ 
\\ \hline
  $p_1^*$ & $W_{1 1}$ & $\rho(-2\delta+2\ep_2)$ & $1$ & $\omega_1+\omega_2$
\\ \cline{2-5} 
& $W_{1 2}$ 
& $\rho(-2\delta+2\ep_2)\rho(-2\delta+(\ep_1+\ep_2))
   \rho(-2\delta+2\ep_1)\rho(-(\ep_1-\ep_2))$
& $-\psi_0^-(q^3 t_0t_1)$ & $\omega_2$ 
\\ \hline
  $p_2^*$ & $W_2$
& $\rho(-2\delta+2\ep_2)\rho(-2\delta+(\ep_1+\ep_2))
   \rho(-2\delta+2\ep_1)\rho(-(\ep_1-\ep_2))$
& $\psi_0^-(q t_0t_1)$ & $\omega_2$ 
\end{tabular}
\caption{Colored alcove walks in the case $v=s_2s_1s_2$}
\label{tab:2:s_2s_1s_2}
\end{table}

The claim is now obtained by summing the terms $A_p B_p C_p P_{-w_0\wtt(p)}(x)$.
\end{proof}

\begin{figure}[htbp]
\includegraphics[bb= 0 0 414 377]{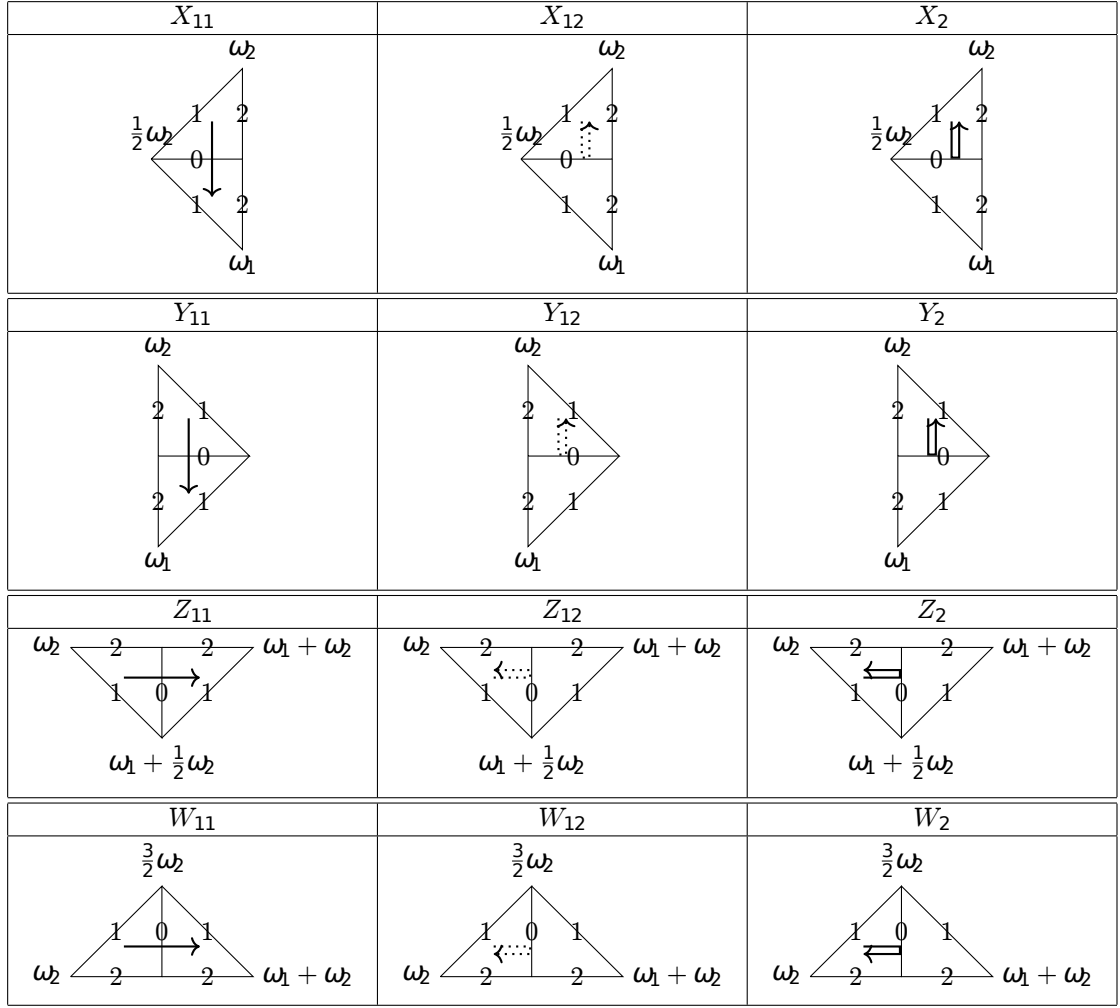}
\caption{Colored alcove walks in Proposition \ref{prp:rk2}}
\label{fig:4.4}
\end{figure}

\newpage

\subsection{Future works}

As a concluding remark, 
let us give some problems which we would like to explore in future works.
\begin{itemize}[nosep]
\item 
Specialization of Theorem \ref{thm:LR} to 
Macdonald polynomials of type $B_n$, $C_n$ and $BC_n$.
\item 
Simplified Pieri formulas and specialization to $B_n$, $C_n$ and $BC_n$ types.
\item 
Tableau-type formulas for Koornwinder polynomials from Theorem \ref{thm:LR}.
\end{itemize}

\subsection*{Acknowledgements}

The author would like to thank the thesis adviser Shintaro Yanagida 
for the guidance and careful reading of the manuscript. 
He would also like to thank Masatoshi Noumi for the explanation 
on the Macdonald-Cherednik theory and Koornwinder polynomials 
through the master course in Kobe University,
and also thank Satoshi Naito for his important comments.

\end{document}